\documentclass[reqno]{amsart}
\usepackage{amsmath}
\usepackage{amsthm}
\usepackage{amssymb}
\usepackage{lineno}

\input xy 
\xyoption{all}

\numberwithin{equation}{section}
\renewcommand{\marginpar}[2][]{}

\usepackage{hyperref}
\newcommand{\GCH}{{\rm GCH}}
\newcommand{\SCH}{{\rm SCH}}
\newcommand{\ZFC}{{\rm ZFC}}

\newcommand{\ORD}{\mathop{{\rm ORD}}}

\renewcommand{\emptyset}{\varnothing}

\renewcommand{\P}{{\mathbb P}}

\newcommand{\Q}{{\mathbb Q}}

\newcommand{\K}{{\mathbb K}}

\newcommand{\Ult}{\mathop{\rm Ult}}
\newcommand{\forces}{\Vdash}

\newcommand{\restrict}{\upharpoonright}

\newcommand{\<}{\langle}
\renewcommand{\>}{\rangle}
\newcommand{\elemsub}{\prec}

\newcommand{\st}{:}

\newcommand{\defn}{\mathop{\rm def}}

\newcommand{\ot}{\mathop{\rm ot}\nolimits}
\newcommand{\cf}{\mathop{\rm cf}}

\newcommand{\id}{\mathop{\rm id}}

\newcommand{\crit}{\mathop{\rm crit}}
\newcommand{\rk}{\mathop{\rm rk}}

\newcommand{\NS}{{\mathop{\rm NS}}}

\newcommand{\WC}{{\mathop{\rm WC}}}
\newcommand{\Refl}{{\mathop{\rm Refl}}}
\newcommand{\NWC}{{\mathop{\rm NWC}}}
\newcommand{\NSS}{{\mathop{\rm NSS}}}

\newcommand{\Sk}{{\mathop{\rm Sk}}}
\newcommand{\NUP}{{\mathop{\rm NUP}}}

\newcommand{\wcdiamond}{\diamondsuit^{\text{wc}}}

\newcommand{\diagonalunion}{\bigtriangledown}

\renewcommand{\and}{\mathop{\&}}

\newtheorem{theorem}{Theorem}[section]
\newtheorem{lemma}[theorem]{Lemma}
\newtheorem{corollary}[theorem]{Corollary}
\newtheorem{proposition}[theorem]{Proposition}

\newtheorem*{theorem12}{Theorem 1.2}
\newtheorem*{theorem14}{Theorem 1.4}

\theoremstyle{definition}
\newtheorem{question}[theorem]{Question}
\newtheorem{remark}[theorem]{Remark}

\newtheorem{definition}[theorem]{Definition}

\thanks{The author would like to thank Hiroshi Sakai for pointing out that Schanker \cite{MR2989393} had already answered the previous version of Question \ref{question_GCH_and_near_supercompactness} and for clarifying several related issues. The author also thanks Piotr Koszmider and Moti Gitik for considering several questions.}

\date{\today}

\begin{document}

%\linenumbers

\title{Characterizations of the weakly compact ideal on $P_\kappa\lambda$}

\author[Brent Cody]{Brent Cody}
\address[Brent Cody]{ 
Virginia Commonwealth University,
Department of Mathematics and Applied Mathematics,
1015 Floyd Avenue, PO Box 842014, Richmond, Virginia 23284, United States
} 
\email[B. ~Cody]{bmcody@vcu.edu} 
\urladdr{http://www.people.vcu.edu/~bmcody/}

\begin{abstract}
Hellsten \cite{MR2026390} gave a characterization of $\Pi^1_n$-indescribable subsets of a $\Pi^1_n$-indescribable cardinal in terms of a natural filter base: when $\kappa$ is a $\Pi^1_n$-indescribable cardinal, a set $S\subseteq\kappa$ is $\Pi^1_n$-indescribable if and only if $S\cap C\neq\emptyset$ for every $n$-club $C\subseteq \kappa$. We generalize Hellsten's characterization to $\Pi^1_n$-indescribable subsets of $P_\kappa\lambda$, which were first defined by Baumgartner. After showing that under reasonable assumptions the $\Pi^1_0$-indescribability ideal on $P_\kappa\lambda$ equals the minimal \emph{strongly} normal ideal $\text{NSS}_{\kappa,\lambda}$ on $P_\kappa\lambda$, and is not equal to $\NS_{\kappa,\lambda}$ as may be expected, we formulate a notion of $n$-club subset of $P_\kappa\lambda$ and prove that a set $S\subseteq P_\kappa\lambda$ is $\Pi^1_n$-indescribable if and only if $S\cap C\neq\emptyset$ for every $n$-club $C\subseteq P_\kappa\lambda$. We also prove that elementary embeddings considered by Schanker \cite{MR2989393} witnessing \emph{near supercompactness} lead to the definition of a normal ideal on $P_\kappa\lambda$, and indeed, this ideal is equal to Baumgartner's ideal of non--$\Pi^1_1$-indescribable subsets of $P_\kappa\lambda$. Additionally, as applications of these results we answer a question of Cox-L\"ucke \cite{MR3620068} about $\mathcal{F}$-layered posets, provide a characterization of $\Pi^m_n$-indescribable subsets of $P_\kappa\lambda$ in terms of generic elementary embeddings, prove several results involving a two-cardinal weakly compact diamond principle and observe that a result of Pereira \cite{MR3640048} yeilds the consistency of the existence of a $(\kappa,\kappa^+)$-semimorasses $\mu\subseteq P_\kappa\kappa^+$ which is $\Pi^1_n$-indescribable for all $n<\omega$.
\end{abstract}



\maketitle

%\tableofcontents

%\noindent \fbox{\begin{minipage}{\textwidth}TO DO:\begin{itemize}\item Add general remark about results for $P_\kappa\lambda$ transferring to $P_{\kappa_x}x$ or $P_\kappa A$.\item Add short argument for adding a weakly compact semimorass.\item Add proof of $1$-club shooting result.\end{itemize}\cite{MR962544} \cite{MR0409186} \cite{MR1635559} \cite{MR1805000} \cite{MR2846027} (defines something called $\NSS_{\kappa,\lambda}$) \cite{MR2989393} \cite{MR1111312} \cite{MR808767} \cite{MR962544} \cite{MR2215786} \cite{MR3509823}\end{minipage}}

% TO DO
% - Remove Remark 3.4 and just make everything more general.
% - [Abe98], page 267, every normal ideal extending $\NSS_{\kappa,\lambda}$ is strongly normal. Is this right?
% - Need to generalize Theorem 4.11 for use in the proof of Theorem 1.2. This suggests that we should verify that one has a strongly normal ideal on $P_\kappa A$ consisting of non-indescribable sets.

\section{Introduction}\label{section_introduction}
%\section{Preliminaries}

Recall that a set $W\subseteq\kappa$ is \emph{weakly compact} if and only if for every $A\subseteq\kappa$ there is a transitive $M\models\ZFC^-$ with $\kappa,A,W\in M$ and $M^{<\kappa}\subseteq M$, there is a transitive $N$ and there is an elementary embedding $j:M\to N$ with critical point $\kappa$ such that $\kappa\in j(W)$. It is well known that $\kappa$ is weakly compact (as a subset of itself) if and only if the collection $\NWC_\kappa=\{X\subseteq\kappa\st \text{$X$ is not weakly compact}\}$ is a normal ideal on $\kappa$, which we refer to as \emph{the weakly compact ideal on $\kappa$}. Baumgartner \cite[Section 2]{MR0540770} showed that assuming $\kappa^{<\kappa}=\kappa$, a set $W\subseteq\kappa$ is weakly compact if and only if it is \emph{$\Pi^1_1$-indescribable}, meaning that for every $\Pi^1_1$-formula $\varphi$ and every $R\subseteq V_\kappa$, if $(V_\kappa,\in, R)\models\varphi$ then there exists $\alpha\in W$ such that $(V_\alpha,\in,R\cap V_\alpha)\models\varphi$. Thus,
\[\NWC_\kappa=\Pi^1_1(\kappa)=_{\defn}\{X\subseteq\kappa\st\text{$X$ is not $\Pi^1_1$-indescribable}\}.\]
Sun \cite{MR1245524} proved that the ideal $\NWC_\kappa$ can be characterized in terms of a natural filter base as follows. A set $C\subseteq\kappa$ is called \emph{$1$-club} if and only if $C\in\NS_\kappa^+$ and whenever $\alpha<\kappa$ is inaccessible and $C\cap\alpha\in\NS_{\alpha}^+$ we have $\alpha\in C$. Sun's characterization of weakly compact sets states that, assuming $\kappa$ is a weakly compact cardinal, a set $W\subseteq\kappa$ is weakly compact if and only if $W\cap C\neq\emptyset$ for every $1$-club $C\subseteq\kappa$. Since a set $S\subseteq\kappa$ is $\Pi^1_0$-indescribable if and only if $\kappa$ is inaccessible and $S\in\NS_\kappa^+$, it follows that for $\kappa$ inaccessible $\NS_\kappa^+=\Pi^1_0(\kappa)^+=\{S\subseteq\kappa\st\text{$S$ is first-order indescribable}\}$ and we can restate Sun's characterization as: for $\kappa$ weakly compact, a set $W\subseteq\kappa$ is weakly compact (or equivalently $\Pi^1_1$-indescribable) if and only if
\begin{align}
(\forall C\in \Pi^1_0(\kappa)^+)((\forall\alpha<\kappa)(C\cap\alpha\in \Pi^1_0(\alpha)^+\implies \alpha\in C)\implies C\cap W\neq\emptyset).\label{equation_sun}
\end{align}

In this article we prove similar results for the weakly compact ideal and the $\Pi^1_n$-indescribability ideals on $P_\kappa\lambda$, which apparently\footnote{Baumgartner's handwritten notes seem to be unavailable.} were first defined by Baumgartner in \cite{Baum}, and have since been studied in \cite{MR808767},  \cite{MR1111312}, \cite{MR1635559}, \cite{MR2846027},  \cite{MR3066742} and \cite{MR3304020}. In \cite{Baum}, Baumgartner defined a notion of $\Pi^m_n$-indescribability for subsets of $P_\kappa\lambda$ using a natural $P_\kappa\lambda$-version of the cumulative hierarchy (see Section \ref{section_indescribability} below), which gives rise to the $\Pi^m_n$-indescribability ideal on $P_\kappa\lambda$
\[\Pi^m_n(\kappa,\lambda)=\{X\subseteq P_\kappa\lambda\st \text{$X$ is not $\Pi^m_n$-indescribable}\}.\] 
Abe \cite{MR1635559} showed that when $P_\kappa\lambda$ is $\Pi^m_n$-indescribable the ideal $\Pi^m_n(\kappa,\lambda)$ is normal. In light of the version of Sun's characterization of weakly compact subsets of $\kappa$ in (\ref{equation_sun}), it seems natural to attempt to give a similar characterization for Baumgartner's notion of $\Pi^1_1$-indescribability for subsets of $P_\kappa\lambda$. We will show that when $\kappa$ is inaccessible, the ideal $\Pi^1_0(\kappa,\lambda)$ of non--$\Pi^1_0$-indescribable (i.e. non--first-order indescribable) subsets of $P_\kappa\lambda$ is equal to the minimal \emph{strongly normal ideal} $\NSS_{\kappa,\lambda}$ of \emph{non strongly stationary} subsets of $P_\kappa\lambda$ (see Section \ref{section_strong_stationarity} below) and is \emph{not} equal to $\NS_{\kappa,\lambda}$ as may be expected. The fact that $\Pi^1_0(\kappa,\lambda)=\NSS_{\kappa,\lambda}$ and the fact that Sun's characterization (\ref{equation_sun}) holds, suggests that the correct notion of ``$1$-club subset of $P_\kappa\lambda$'' needed to generalize Sun's characterization to $P_\kappa\lambda$ should be stated using $\NSS_{\kappa,\lambda}^+$ instead of $\NS_{\kappa,\lambda}^+$. Recall that for $x\in P_\kappa\lambda$ we define $\kappa_x=_{\defn}|x\cap\kappa|$.
\begin{definition}\label{definition_1_club}
We say that $C\subseteq P_\kappa\lambda$ is \emph{$1$-club} if and only if 
\begin{enumerate}
\item $C\in \NSS_{\kappa,\lambda}^+$ and
\item $C$ is \emph{$1$-closed}, that is, for every $x\in P_\kappa\lambda$, if $\kappa_x$ is an inaccessible cardinal and $C\cap P_{\kappa_x}x\in\NSS_{\kappa_x,x}^+$ then $x\in C$.
\end{enumerate}
\end{definition}
\noindent In Section \ref{section_n_clubs}, we generalize Sun's characterization of $\Pi^1_1$-indescribable subsets of $\kappa$ by showing that the notion of $1$-club subset of $P_\kappa\lambda$ in Definition \ref{definition_1_club} can indeed be used to characterize the $\Pi^1_1$-indescribable subsets of $P_\kappa\lambda$. In fact, in Section \ref{section_n_clubs}, we develop a notion of $n$-club subset of $P_\kappa\lambda$ where $n<\omega$ and prove the following.

\begin{theorem}\label{theorem_n_club}
Suppose $\kappa\leq\lambda$ are cardinals with $\lambda^{<\kappa}=\lambda$, $n<\omega$ and $P_\kappa\lambda$ is $\Pi^1_n$-indescribable. Then $S\subseteq P_\kappa\lambda$ is $\Pi^1_n$-indescribable if and only if for all $n$-clubs $C\subseteq P_\kappa\lambda$ we have $S\cap C\neq\emptyset$.
\end{theorem}

In Section \ref{section_n_clubs}, we also prove that Baumgartner's $\Pi^1_1$-indescribable subsets of $P_\kappa\lambda$ can be characterized using elementary embeddings which resemble the usual elementary embeddings witnessing the weak compactness of subsets of $\kappa$. Recall that, for cardinals $\kappa\leq\lambda$, $\kappa$ is \emph{$\lambda$-supercompact} if and only if there is an elementary embedding $j:V\to M$ with critical point $\kappa$ such that $j(\kappa)>\lambda$ and $j"\lambda\in M$. Such embeddings can be assumed to be ultrapowers by normal fine $\kappa$-complete ultrafilters on $P_\kappa\lambda$, in which case $M^\lambda\cap V\subseteq M$. Schanker \cite{MR2989393}  fused the notions of weak compactness and $\lambda$-supercompactness as follows: $\kappa$ is said to be \emph{nearly $\lambda$-supercompact} if for every $A\subseteq\lambda$ there is a transitive $M\models\ZFC^-$ with $\lambda,A\in M$ and $M^{<\kappa}\cap V\subseteq M$, there is a transitive $N$ and an elementary embedding $j:M\to N$ with critical point $\kappa$ such that $j(\kappa)>\lambda$ and $j"\lambda\in N$. As observed by Schanker, it is clear that if $\kappa$ is $\lambda$-supercompact then $\kappa$ is nearly $\lambda$-supercompact and the converse is not true in general; for example, the least cardinal $\kappa$ which is nearly $\kappa^+$-supercopmact is not $\kappa^+$-supercompact. Even though $\kappa$ is supercompact if and only if $\kappa$ is nearly $\lambda$-supercompact for every $\lambda\geq\kappa$, Schanker proved that for any fixed $\lambda\geq\kappa$, $\kappa$ being nearly $\lambda$-supercompact need not imply that $\kappa$ is measurable. For example, Schanker proved that if $\kappa$ is nearly $\kappa^+$-supercompact and $\GCH$ holds, then there is a cofinality-preserving forcing extension in which $\kappa$ remains nearly $\kappa^+$-supercompact and $\GCH$ fails first at $\kappa$. Furthermore, if $\kappa$ is $\lambda$-supercompact and $\GCH$ holds then there is a cofinality-preserving forcing extension \cite{MR3372604} in which $\kappa$ is the least weakly compact cardinal and $\kappa$ is nearly $\lambda$-supercompact.\footnote{$\GCH$ must fail at $\kappa$ is such an extension.}

We prove that the elementary embeddings considered by Schanker in \cite{MR2989393} lead to a normal ideal on $P_\kappa\lambda$ as follows, and indeed this ideal is equal to Baumgartner's ideal $\Pi^1_1(\kappa,\lambda)$. %when $\kappa$ is nearly $\lambda$-supercompact.
\begin{definition}\label{definition_weakly_compact}
We say that a set $W\subseteq P_\kappa\lambda$ is \emph{weakly compact}\footnote{We prefer this terminology to saying that ``$W$ is nearly $\lambda$-supercompact'' because we will prove that $W\subseteq P_\kappa\lambda$ is weakly compact if and only if $W$ is $\Pi^1_1$-indescribable, and thus this terminology conforms to more of the existing literature.} if and only if for every $A\subseteq\lambda$ there is a transitive $M\models\ZFC^-$ with $\lambda,A,W\in M$, a transitive $N$ and an elementary embedding $j:M\to N$ with critical point $\kappa$ such that $j(\kappa)>\lambda$ and $j"\lambda\in j(W)$. The \emph{weakly compact ideal on $P_\kappa\lambda$} is defined to be
\[\NWC_{\kappa,\lambda}=\{X\subseteq P_\kappa\lambda\st \text{$X$ is not weakly compact}\}.\]
\end{definition} 

\begin{theorem}\label{theorem_characterization_embedding}
Suppose $\kappa\leq\lambda$ are cardinals such that $\kappa$ is inaccessible and $\lambda^{<\kappa}=\lambda$. Then a set $W\subseteq P_\kappa\lambda$ is $\Pi^1_1$-indescribable if and only if it is weakly compact.
\end{theorem}

In Section \ref{section_applications}, we provide several applications. In Section \ref{section_cox_lucke}, we answer a question of Cox and L\"ucke \cite{MR3620068}. Before stating the question, let us review some terminology from \cite{MR3620068}. For partial orders $\Q\subseteq\P$, we say that $\Q$ is a \emph{regular suborder of $\P$} if the inclusion map preserves incompatibility and maximal antichains in $\Q$ are also maximal in $\P$. Given a partial order $\P$, we let $\text{Reg}_\kappa(\P)$ denote the collection of all regular suborders of $\P$ of cardinality less than $\kappa$. In \cite{MR3620068}, the authors consider various properties of partial orders that imply $\text{Reg}_\kappa(\P)$ is \emph{large} in a certain sense. For example, suppose $\kappa$ is a regular uncountable cardinal, a partial order $\P$ is called \emph{$\kappa$-stationarily layered} if $\text{Reg}_\kappa(\P)$ is stationary in $P_\kappa\P$. Among other things, Cox-L\"ucke showed \cite[Theorem 1.8]{MR3620068} that such properties can be used to characterize weakly compact cardinals: $\kappa$ is a weakly compact cardinal if and only if every partial order $\P$ satisfying the $\kappa$-chain condition is $\kappa$-stationarily layered. Cox and L\"ucke also consider another notion of \emph{largeness} of $\text{Reg}_\kappa(\P)$: a partial order $\P$ is \emph{$\mathcal{F}$-layered} if it has cardinality at most $\lambda$ and $\{a\in P_\kappa\lambda\st s[a]\in\text{Reg}_\kappa(\P)\}\in \mathcal{F}$ holds for every surjection $s:\lambda\to \P$. Question 7.4 of \cite{MR3620068} states, assuming $(\kappa^+)^{<\kappa}=\kappa^+$, ``Let $\kappa$ be an inaccessible cardinal such that there is a normal filter $\mathcal{F}$ on $P_\kappa\kappa^+$ with the property that every partial order of cardinality $\kappa^+$ that satisfies the $\kappa$-chain condition is $\mathcal{F}$-layered. Must $\kappa$ be a measurable cardinal?'' By generalizing the work of Schanker \cite{MR2989393}, we show that the answer is no by proving the following.

\begin{theorem}\label{theorem_answer}
Suppose $P_\kappa\lambda$ is weakly compact, $\GCH$ holds and $\lambda^{<\lambda}=\lambda$. There is a cofinality-preserving forcing extension $V[G]$ in which  
\begin{enumerate}
\item $(P_\kappa\lambda)^{V[G]}$ is weakly compact and hence the filter $\mathcal{F}=(\NWC_{\kappa,\lambda}^*)^{V[G]}$ is normal and nontrivial,
\item every partial order of cardinality $\lambda$ that satisfies the $\kappa$-c.c. is $\mathcal{F}$-layered,
\item $\kappa$ is not measurable and
\item $\lambda^{<\kappa}=\lambda$.
\end{enumerate} 
\end{theorem}

In Section \ref{section_generic_embeddings}, we consider properties of generic ultrapowers by the $\Pi^m_n$-indescribability ideals on $P_\kappa\lambda$. Indeed, we give a characterization of $\Pi^m_n$-indescribable subsets of $P_\kappa\lambda$ in terms of generic elementary embeddings.

In Section \ref{section_diamond}, generalizing similar principles considered by Hellsten \cite{MR2026390}, we use the weakly compact ideal $\NWC_{\kappa,\lambda}$ to formulate a two-cardinal weakly compact diamond principle as follows.

\begin{definition} \label{definition_weakly_compact_diamond}
Suppose $W\in \NWC_{\kappa,\lambda}^+$. We say that \emph{weakly compact diamond holds on $W$} and write $\wcdiamond_{\kappa,\lambda}(W)$ if and only if there is a sequence $\<a_z\subseteq z\st z\in P_\kappa\lambda\>$ such that for every $A\subseteq \lambda$ we have $\{z \in W\st a_z=A\cap z\}\in \NWC_{\kappa,\lambda}^+$. When $W=P_\kappa\lambda$ we write simply $\wcdiamond_{\kappa,\lambda}$ instead of $\wcdiamond_{\kappa,\lambda}(P_\kappa\lambda)$.
\end{definition}
As an application of the $1$-club characterization of weakly compact subsets of $P_\kappa\lambda$ obtained by combining Theorem \ref{theorem_n_club} and Theorem \ref{theorem_characterization_embedding}, we prove that for any $W\in\NWC_{\kappa,\lambda}^+$, if $\kappa$ is $\lambda$-supercompact then $\wcdiamond_{\kappa,\lambda}(W)$ holds. We also show that, assuming $\GCH$, there is a natural way to force $\wcdiamond_{\kappa,\lambda}(W)$ without collapsing cofinalities from the assumption that $P_\kappa\lambda$ is weakly compact and $\lambda^{<\lambda}=\lambda$.

In Section \ref{section_semimorasses}, we use a result of Pereira \cite{MR3640048} to show that if $\kappa$ is $\kappa^+$-supercompact and $\GCH$ holds then there is a cofinality-preserving forcing extension in which there is a $(\kappa,\kappa^+)$-semimorass $\mu\subseteq (\kappa,\kappa^+)$ which is $\Pi^1_n$-indescribable for all $n<\omega$.

We close the paper with a discussion of several open questions concerning reflection properties of weakly compact sets $W\subseteq P_\kappa\lambda$ and generalizations of club shooting forcing to the context of the weakly compact ideal on $P_\kappa\lambda$.

\section{Preliminaries on strongly normal ideals and strong stationarity}\label{section_strong_stationarity}

Throughout this section we assume $\kappa\leq\lambda$ are cardinals, $\kappa$ is a regular cardinal and $X$ is a set of ordinals. Recall that an ideal $I$ on $P_\kappa X$ is \emph{normal} if for every $S\in I^+$ and every function $f:P_\kappa X\to X$ with $\{x\in S \st f(x)\in x\}\in I^+$ there is a $T\in P(S)\cap I^+$ such that $f\restrict T$ is constant. Equivalently, an ideal $I$ on $P_\kappa X$ is normal if and only if for every $\{Z_x\st x\in X\}\subseteq I$ the set $\diagonalunion_{x\in X}Z_x=_{\defn}\{y\in P_\kappa X\st\text{$y\in Z_x$ for some $x\in y$}\}$ is in $I$ (see \cite[Proposition 2.19]{MR2768692}). An ideal $I$ on $P_\kappa\lambda$ is \emph{fine} if and only if $\widetilde{\{\alpha\}}=_{\defn}\{x\in P_\kappa\lambda\st \alpha\in x\}\in I^*$ for every $\alpha<\lambda$. Jech \cite{MR0325397} generalized the notion of closed unbounded and stationary subsets of cardinals to subsets $P_\kappa\lambda$. Recall that a set $C\subseteq P_\kappa\lambda$ is \emph{club in $P_\kappa\lambda$} if (1) for every $x\in P_\kappa\lambda$ there is a $y\in C$ with $x\subseteq y$ and (2) whenever $X\subseteq C$ is directed under the ordering $\subsetneq$ and $|X|<\kappa$ we have $\bigcup X\in C$. A set $S\subseteq P_\kappa\lambda$ is \emph{stationary} if $S\cap C\neq\emptyset$ for all clubs $C\subseteq P_\kappa\lambda$. Jech proved that the collection $\NS_{\kappa,\lambda}=\{X\subseteq P_\kappa\lambda\st\text{$X$ is not stationary in $P_\kappa\lambda$}\}$ is a normal fine $\kappa$-complete ideal on $P_\kappa\lambda$. Carr \cite{MR667297} proved that, when $\kappa$ is a regular cardinal, the nonstationary ideal $\NS_{\kappa,\lambda}$ is the \emph{minimal} normal fine $\kappa$-complete ideal on $P_\kappa\lambda$.

%[Insert remark about the structure $(P_\kappa\lambda,<)$, definitions, etc.] 

When considering ideals on $P_\kappa\lambda$ or $P_\kappa X$ for $\kappa$ inaccessible, it is quite fruitful to work with a different notion of closed unboundedness obtained by replacing the structure $(P_\kappa\lambda,\subseteq)$ with a different one. For $x\in P_\kappa X$ we define $\kappa_x=|x\cap\kappa|$ and we define an ordering $(P_\kappa X,\sqsubset)$ by letting 
\[\text{$x\sqsubset y$ if and only if $x\in P_{\kappa_y}y$}.\] Given a function $f:P_\kappa X\to P_\kappa X$ we let
\[C_f=_{\defn}\{x\in P_\kappa X\st x\cap\kappa\neq\emptyset\land f[P_{\kappa_x}x]\subseteq P_{\kappa_x}x\}.\] 
We say that a set $C\subseteq P_\kappa X$ is \emph{weakly closed unbounded} if there is an $f$ such that $C= C_f$. Note that it is straightforward to see that when $\kappa$ is inaccessible, every club $C\subseteq P_\kappa X$ contains a weak club (see Lemma \ref{lemma_jech_clubs_are_weak_clubs} below). However, in general, it is not the case that every weak club contains a club (this follows from Corollary \ref{corollary_NS_not_NSS} below). A set $S\subseteq P_\kappa X$ is called \emph{strongly stationary} if for every $f$ we have $C_f\cap S\neq\emptyset$. An ideal $I$ on $P_\kappa X$ is \emph{strongly normal} if for any $S\in I^+$ and function $f:P_\kappa X\to P_\kappa X$ such that $f(x)\sqsubset x$ for all $x\in X$ there is $Y\in P(X)\cap I^+$ such that $f\restrict Y$ is constant. It follows easily that an ideal $I$ on $P_\kappa X$ is strongly normal if and only if for any $\{X_a\st a\in P_\kappa X\}\subseteq I$ the set $\diagonalunion_{\sqsubset} X_a=_{\defn}\{x\in P_\kappa X\st \text{$x\in X_a$ for some $a\sqsubset x$}\}$ is in $I$. Note that an easy argument shows that if $\kappa$ is $\lambda$-supercompact then the prime ideal dual to a normal fine ultrafilter on $P_\kappa\lambda$ is strongly normal. Matet \cite{MR954259} showed that if $\kappa$ is Mahlo then the collection of non--strongly stationary sets
\[\NSS_{\kappa,X}=_{def}\{X\subseteq P_\kappa X\st\text{$\exists f:P_\kappa X\to P_\kappa X$ such that $X\cap C_f=\emptyset$}\}\]
is the minimal strongly normal ideal on $P_\kappa X$. Improving this, Carr, Levinski and Pelletier obtained the following.

\begin{theorem}[Carr-Levinski-Pelletier \cite{MR1074449}]\label{cpl_minimal_strongly_normal_ideal} Suppose $\kappa$ is a regular cardinal and $X$ is a set of ordinals with $\kappa\leq |X|$. There is a strongly normal ideal on $P_\kappa X$ if and only if $\kappa$ is Mahlo or $\kappa=\mu^+$ where $\mu^{<\mu}=\mu$; moreover, in this case $\NSS_{\kappa, X}$ is the minimal such ideal. \marginpar{\tiny We need a slight generalization which is that $P_{\kappa_x}x$ carries a strongly normal ideal if and only if $\kappa_x$ is Mahlo or $\kappa_x=\mu^+$ for some $\mu$ with $\mu^{<\mu}=\mu$.}
\end{theorem}

In these cases, since every strongly normal ideal on $P_\kappa\lambda$ is normal, we have
\[\NS_{\kappa,\lambda}\subseteq\NSS_{\kappa,\lambda}.\]
The following lemma, due to Zwicker (see the discussion on page 61 of \cite{MR1074449}), shows that if $\kappa$ is weakly inaccessible the previous containment is strict. We include a proof for the reader's convenience.
\begin{lemma}[Zwicker]
If $\kappa$ is weakly inaccessible then $\NS_{\kappa,\lambda}$ is not strongly normal. \marginpar{\tiny Do we have strict containment when $\kappa$ is a successor? I.e. is $\NS_{\kappa,\lambda}$ not strongly normal when $\kappa$ is the successor of $\mu$ where $\mu^{<\mu}=\mu$?}
\end{lemma}

\begin{proof}
Let $A=\{x\in P_\kappa\lambda\st\text{$x\cap\kappa$ is an uncountable cardinal with $\cf(x\cap\kappa)=\omega$}\}$. First we show that $A$ is a stationary subset of $P_\kappa\lambda$. Let $C\subseteq P_\kappa\lambda$ be a club and recall that the set $C^*=\{x\in P_\kappa\lambda\st x\cap \kappa\in\kappa\}$ is a club subset of $P_\kappa\lambda$. We inductively define a sequence $\<x_i\st i<\omega\>$ with $x_i\in C\cap C^*$ as follows. Let $\kappa_0=\omega_1$ and choose $x_0\in C\cap C^*$ with $\kappa_0\subseteq x_0$. Let $\kappa_{i+1}>x_i\cap \kappa$ and choose $x_{i+1}\in C\cap C^*$ with $\kappa_{i+1}\subsetneq x_{i+1}$. Now let $x_\omega=\bigcup_{i<\omega}x_i$ and notice that $x_\omega\in C\cap C^*\cap A$.

Now define $F:A\to P_\kappa\lambda$ by letting $F(x)$ be some countable and cofinal subset of $x\cap\kappa$. Then $F(x)\sqsubset x$ for all $x\in A$ and $F$ is not constant on any stationary subset of $A$.
\end{proof} 

\begin{corollary}\label{corollary_NS_not_NSS}
If $\kappa$ is Mahlo then $\NSS_{\kappa,\lambda}$ is nontrivial and $\NS_{\kappa,\lambda}\subsetneq\NSS_{\kappa,\lambda}$.
\end{corollary}

\begin{lemma}\label{lemma_jech_clubs_are_weak_clubs}
Suppose $\kappa$ is an inaccessible cardinal and $X$ is a set of ordinals with $\kappa\leq|X|$. If $C$ is a club subset of $P_\kappa X$ and $f:P_\kappa X\to P_\kappa X$ is such that $z\subsetneq f(z)\in C$ for every $z\in P_\kappa X$, then
\[C_f=\{x\in P_\kappa X\st x\cap\kappa\neq\emptyset\land f[P_{\kappa_x}x]\subseteq P_{\kappa_x}x\}\]
is a subset of $C$. \marginpar{\tiny We need this for $P_{\kappa_x}x$ too!}
 %Hence $\NS_{\kappa,\lambda}\subseteq\NSS_{\kappa,\lambda}$. 
\end{lemma}

\begin{proof}
Suppose $C$ is a club subset of $P_\kappa X$ and $f$ is as in the statement of the lemma. Suppose $x\in P_\kappa X$ and $f[P_{\kappa_x}x]\subseteq P_{\kappa_x}x$. It follows that $C\cap P_{\kappa_x}x$ is directed since if $y,z\in C\cap P_{\kappa_x}x$ then $y\cup z\in P_{\kappa_x}x$ and hence $y\cup z\subsetneq f(y\cup z)\in C\cap P_{\kappa_x}x$. Since $C\cap P_{\kappa_x}x$ is a directed subset of $C$ with size at most $|x|^{<\kappa_x}<\kappa$, it follows that $x=\bigcup (C\cap P_{\kappa_x}x)\in C$. Thus $C_f\subseteq C$.
\end{proof}

The next lemma shows that $\NSS_{\kappa,\lambda}$ can be obtained by restricting $\NS_{\kappa,\lambda}$ to a particular stationary sets

\begin{lemma}[\cite{MR1074449}, Corollary 3.3]
If $\lambda^{<\kappa}=\lambda$ and $\NSS_{\kappa,\lambda}$ is nontrivial, then for any bijection $c:P_\kappa\lambda\to \lambda$, 
\[\NSS_{\kappa,\lambda}=\NS_{\kappa,\lambda}\restrict S_c\] 
where $S_c=\{x\in P_\kappa\lambda\st c[P_{\kappa_x}x]\subseteq x\}$.
\end{lemma}

\section{Elementary embeddings and the weakly compact ideal on $P_\kappa\lambda$}

There are many ways to characterize the weakly compact subsets of $P_\kappa\lambda$ from Definition \ref{definition_weakly_compact} using elementary embeddings. Indeed, all six characterizations of the near $\lambda$-supercompactness of a cardinal $\kappa$ given in \cite{MR2989393} can be generalized to provide characterizations of weakly compact subsets of $P_\kappa\lambda$. Here we summarize the pertinent characterizations without proof; note that the proof is very similar to that of \cite[Theorem 1.4]{MR2989393}.

\begin{lemma}[Schanker]\label{lemma_basic_embedding_characterizations}
For cardinals $\kappa\leq\lambda$ with $\lambda^{<\kappa}=\lambda$ and $W\subseteq P_\kappa\lambda$, the following are equivalent.
\begin{enumerate}
\item $W$ is a weakly compact subset of $P_\kappa\lambda$; in other words, for every $A\subseteq\lambda$ there is a transitive $M\models\ZFC^-$ with $\lambda,A,W\in M$ and $M^{<\kappa}\cap V\subseteq M$, a transitive $N$ and an elementary embedding $j:M\to N$ with critical point $\kappa$ such that $j(\kappa)>\lambda$ and $j"\lambda\in j(W)$.
\item For all $\delta\geq\kappa$ and every transitive $M\models\ZFC^-$ of size $\lambda$ with $\lambda,W\in M$ and $M^{<\delta}\cap V\subseteq M$, there is a transitive $N$ of size $\lambda$ with $N^{<\delta}\cap V\subseteq N$ and $P(\lambda)^M\subseteq N$ and an elementary embedding $j:M\to N$ with critical point $\kappa$ such that $j(\kappa)>\lambda$, $j"\lambda\in j(W)$ and $N=\{j(f)(j"\lambda)\st\text{$f\in M$ is a function with domain $P_\kappa\lambda$}\}.$
\item For every collection $\mathcal{A}$ of at most $\lambda$ subsets of $P_\kappa\lambda$ with $W\in \mathcal{A}$ and every collection $\mathcal{F}$ of at most $\lambda$ functions from $P_\kappa\lambda$ to $\lambda$, there exists a $\kappa$-complete filter $F$ on $P_\kappa\lambda$ such that $W\in F$; $(\forall\alpha<\lambda)(\widetilde{\{\alpha\}}=_{\defn}\{x\in P_\kappa\lambda\st\alpha\in x\}\in F)$; $F$ measures all sets in $\mathcal{A}$, meaning that for all $X\in\mathcal{A}$ either $X\in F$ or $P_\kappa\lambda\setminus X\in F$, and finally, $F$ is $\mathcal{F}$-normal, in the sense that for every $f\in \mathcal{F}$ which is regressive on some set in $F$, there is $\alpha_f<\lambda$ such that $\{x\in P_\kappa\lambda\st f(x)=\alpha_f\}\in F$.
\end{enumerate}
\end{lemma}

By assuming a little bit more about cardinal arithmetic we obtain another characterization which will be useful for forcing arguments.

\begin{lemma}\label{lemma_normal_embedding}
If $\lambda^{<\lambda}=\lambda$ and $W\subseteq P_\kappa\lambda$ then $W$ is a weakly compact subset of $P_\kappa\lambda$ if and only if for every $A\in H(\lambda^+)$ there is a transitive $M\models\ZFC^-$ of size $\lambda$ with $\lambda,A,W\in M$ and $M^{<\lambda}\cap V\subseteq M$, there is a transitive $N$ of size $\lambda$ with $N^{<\lambda}\cap V\subseteq N$ and an elementary embedding $j:M\to N$ with critical point $\kappa$ such that 
\begin{enumerate}
\item $j(\kappa)>\lambda$,
\item $j"\lambda\in j(W)$,
\item $N=\{j(f)(j"\lambda)\st \text{$f\in M$ is a function with domain $P_\kappa\lambda$}\}$ and
\item if $X\in M$ with $|X|^M\leq\lambda$ and $X\in N$ then $j\restrict X\in N$.
%\item for every $x\in P(M)^V$ with $|x|^V\leq\lambda$ we have $j"x\in N$.\footnote{Note that (4) is not entirely trivial because it may be that $x\notin M$ and hence $j(x)$ could be undefined.}
\end{enumerate}
\end{lemma}

\begin{proof}
The reverse direction is easy. For the forward direction, assume $W$ is a weakly compact subset of $P_\kappa\lambda$ and $A\in H(\lambda^+)$. Since $\lambda^{<\lambda}\subseteq\lambda$ we can use an iterative Skolem-hull argument to build a transitive $M\elemsub H(\lambda^+)$ of size $\lambda$ with $\lambda,A,W\in M$ and $M^{<\lambda}\cap V\subseteq M$. Now, applying Lemma \ref{lemma_basic_embedding_characterizations} (3), there is a $\kappa$-complete fine $M$-normal $M$-ultrafilter $F$ with $W\in F$. Let $j:M\to \Ult(M,F)=(^\kappa M\cap M)/F$ be the corresponding ultrapower embedding, which is well-founded since $F$ is $\kappa$-complete. Thus we may identify $\Ult(M,F)$ with its transitive collapse $N$ and obtain $j:M\to N$. To see that (1) -- (3) hold one may apply standard arguments. For (4), suppose $X\in M$ with $|X|^M\leq\lambda$ and $X\in N$. Let $b:\lambda\to X$ be a bijection in $M$. By elementarity $j(b):j(\lambda)\to j(X)$ is a bijection in $N$ and $j(b)[j"\lambda]=j"X$. Furthermore, working in $N$, we may define a function $f$ with domain $X$ such that $f(x)=j(b)(j(b^{-1}(x)))=j(b)(j(b^{-1})(j(x)))=j\restrict X(x)$.
\end{proof}

% For (4), suppose $x\in P(M)^V$ and $|x|^V\leq\lambda$. Let $\vec{x}=\<x_\alpha\st\alpha<\lambda\>\in V$ be an enumeration of the elements of $x$. For each $\alpha<\lambda$ we have $j(x_\alpha)=[c_\alpha]_F=j(f_\alpha)(j"\lambda)$ for some function $f_\alpha\in M$ where $c_\alpha$ is the function with domain $P_\kappa\lambda$ and constant value $x_\alpha$. If $\vec{h}=_{\defn}\<h_\alpha\st\alpha<\lambda\>\in M$ then $j(\vec{h})\restrict j"\lambda\in N$ and hence $\<j(x_\alpha)\st\alpha<\lambda\>=\<j(h_\alpha)(j"\lambda)\st\alpha<\lambda\>\in N$ which implies $j"x\in N$. On the other hand, even if $\vec{h}\notin M$, we do have $\vec{c}=\<c_\alpha\st\alpha<\lambda\>\in M$. Thus $j(\vec{c})\restrict j"\lambda=\<j(c_\alpha)\st\alpha<\lambda\>\in N$ and hence $\<j(c_\alpha)(j"\lambda)\st\alpha<\lambda\>=\<j(x_\alpha)\st\alpha<\lambda\>\in N$. Therefore, $j"x\in N$.

% show that the critical point of $j$ is $\kappa$ and $X\in F$ if and only if $[\id]_F\in j(X)$ for all $X\subseteq P_\kappa\lambda$, where $\id:P_\kappa\lambda\to P_\kappa\lambda$ is the identity function. Furthermore, by fineness and $M$-normality $[\id]_F=j"\lambda$ and hence every element of $N$ is of the form $[f]_F=j(f)([\id]_F)=j(f)(j"\lambda)$ for some $f\in M$. Since $\alpha=[\<\ot(x\cap\alpha)\st x\in P_\kappa\lambda\>]_F$ for $\alpha\leq\lambda$ and $j(\kappa)=[x\mapsto\kappa]_F$, it follows that $j(\kappa)>\lambda$. Since $W\in F$ we see that $j"\lambda\in j(W)$.

The next definition will make negating the definition of weakly compact set easier.

\begin{definition}
Suppose $\kappa\leq\lambda$ are cardinals and $Z\in \NWC_{\kappa,\lambda}$. We say that $A\subseteq \lambda$ \emph{witnesses that $Z$ is not weakly compact} or \emph{witnesses $Z\in \NWC_{\kappa,\lambda}$} if and only if whenever $M\models\ZFC^-$ is transitive with $\lambda,A,Z\in M$ and whenever $N$ is transitive and $j:M\to N$ is an elementary embedding with critical point $\kappa$ such that $j(\kappa)>\lambda$, we must have $j"\lambda\notin j(Z)$; in other words, $A$ being in $M$ guarantees that $j"\lambda\notin j(Z)$.
\end{definition}

\begin{proposition}
If $P_\kappa \lambda$ is weakly compact then the non--weakly compact ideal $\NWC_{\kappa,\lambda}$ is a strongly normal proper ideal.
\end{proposition}

\begin{proof}
Let us show that $\NWC_{\kappa,\lambda}$ is strongly normal; the rest is routine. Suppose $Z_a\in \NWC_{\kappa,\lambda}$ for all $a\in P_\kappa \lambda$ and let $Z=\diagonalunion_{\sqsubset}\{Z_a\st a\in P_\kappa \lambda\}=_{\defn}\{x\in P_\kappa \lambda\st \text{$x\in Z_a$ for some $a\in P_{\kappa_x}x$}\}$. For each $a\in P_\kappa \lambda$ there is some $A_a\subseteq \lambda$ witnessing that $Z_a\in\NWC_{\kappa,\lambda}$. Since $|P_\kappa \lambda|=|\lambda|^{<\kappa}=|\lambda|$ there is a single set $A\subseteq \lambda$ coding all of the $A_a$'s as well as the sequence $\vec{Z}=\<Z_a\st a\in P_\kappa \lambda\>$ in the sense that whenever $M$ is transitive with $A\in M$ then $A_a\in M$ for all $a\in P_\kappa \lambda$ and $\vec{Z}\in M$. Clearly we have that for every $a\in P_\kappa \lambda$ the set $A$ witnesses that $Z_a\in\NWC_{\kappa,\lambda}$. Let us argue that $A$ witnesses that $Z\in \NWC_{\kappa,\lambda}$. Suppose $M\models\ZFC^-$ is transitive of size $\lambda$ with $\lambda,A,Z\in M$, $N$ is transitive and $j:M\to N$ is an elementary embedding with critical point $\kappa$ such that $j(\kappa)>\lambda$. We must argue that $j"\lambda\notin j(Z)$. Since $A\in M$ we have $\vec{Z}\in M$ and we let $j(\vec{Z})=\<\bar{Z}_b\st b\in j(P_\kappa \lambda)\>$. Notice that $\bar{Z}_{j(a)}=j(Z_a)$ for all $a\in P_\kappa \lambda$ by elementarity of $j$. By definition of $Z$,
\[j(Z)=\{x\in j(P_\kappa \lambda)\st \text{$x\in \bar{Z}_b$ for some $b\in P_{j(\kappa)_x} x$}\}\]
where $j(\kappa)_{j"\lambda}=|j"\lambda\cap j(\kappa)|^N=\kappa$. For the sake of contradiction, suppose $j"\lambda\in j(Z)$. Then $j"\lambda\in \bar{Z}_b$ for some $b\in P_{j(\kappa)_{j"\lambda}} j"\lambda= P_\kappa j"\lambda$. Since the critical point of $j$ is $\kappa$ we see that $b=j(a)$ for some $a\in P_\kappa \lambda$ and hence $j"\lambda\in \bar{Z}_{j(a)}=j(Z_a)$ for some $a\in P_\kappa \lambda$. This contradicts the fact that $A$ witnesses $Z_a\in \NWC_{\kappa,\lambda}$.
\end{proof}

Since $\NSS_{\kappa,\lambda}$ is the minimal strongly normal ideal on $P_{\kappa}\lambda$ we obtain.

\begin{corollary}\label{corollary_NSS_contained_in_NWC}
If $P_\kappa\lambda$ is weakly compact then $\NSS_{\kappa,\lambda}\subseteq\NWC_{\kappa,\lambda}$.
\end{corollary}

To see that $\NSS_{\kappa,\lambda}\subsetneq\NWC_{\kappa,\lambda}$ when $P_\kappa\lambda$ is weakly compact, let us consider the set $S=\{x\in P_\kappa\lambda\st |x\cap\kappa|=|x|\}$, which played an important role in various results on almost disjoint partitions of elements of $\NS_{\kappa,\lambda}^+$ (see the discussion around Proposition 25.5 in \cite{Kanamori:Book}).

\begin{proposition}[Proposition 25.5, \cite{Kanamori:Book}]\label{proposition_kanamori_S}
Suppose that $\kappa\leq\lambda$ and $S=\{x\in P_\kappa\lambda\st |x\cap\kappa|=|x|\}$. Then:
\begin{enumerate}
\item[(a)] $S\in\NS_{\kappa,\lambda}^+$
\item[(b)] If $\kappa$ is a successor cardinal then $S\in \NS_{\kappa,\lambda}^*$.
\item[(c)] If $\kappa<\lambda$ and $\kappa$ is $\lambda$-supercompact then $S\notin \NS_{\kappa,\lambda}^*$.
\item[(d)] (Baumgartner) If $X\subseteq S$ is stationary then $X$ can be partitioned into $\lambda$ disjoint stationary sets.
\end{enumerate}
\end{proposition}

Straight forward arguments show that if $\NSS_{\kappa,\lambda}$ is nontrivial, Proposition \ref{proposition_kanamori_S} can be improved by replacing $\NS_{\kappa,\lambda}$ with $\NSS_{\kappa,\lambda}$ in (a) and (b) and by weakening the hypothesis and strengthening the conclusion of (c).

\begin{proposition}
Suppose $\kappa\leq\lambda$ and $S=\{x\in P_\kappa\lambda\st |x\cap\kappa|=|x|\}$. Then the following hold.
\begin{enumerate}
\item[(a)] If $\NSS_{\kappa,\lambda}$ is nontrivial then $S\in \NSS_{\kappa,\lambda}^+$.
\item[(b)] If $\kappa$ is a successor cardinal then $S\in \NSS_{\kappa,\lambda}^*$.
\item[(c)] If $P_\kappa\lambda$ is weakly compact then $S\in\NWC_{\kappa,\lambda}$.
\end{enumerate}
\end{proposition}

\begin{proof}
For (a), first notice that if $\kappa$ is a successor then by Proposition \ref{proposition_kanamori_S} (a), $S\in \NS_{\kappa,\lambda}^*\subseteq \NSS_{\kappa,\lambda}^+$. On the other hand if $\kappa$ is a limit, then by Theorem \ref{cpl_minimal_strongly_normal_ideal}, $\kappa$ is Mahlo since $\NSS_{\kappa,\lambda}$ is nontrivial. Fix $f:P_\kappa\lambda\to P_\kappa\lambda$ and recursively define $x(n)$ and $y(n)$ as follows. Let $y(n+1)=\bigcup f[P_{\kappa_{x(n)}}x(n)]\in P_\kappa\lambda$ and define $x(n+1)=y(n)\cup|y(n)|$. Notice that $x(n+1)\in S$. Now it follows that $x(\omega)=_{\defn}\bigcup_{n<\omega}x(n)\in S$ \marginpar{\tiny Double check this: I think $S$ is closed under countable unions. See proof of Prop. 25.5 (a) in Kanamori, page 343.} and $f[P_{\kappa_{x(\omega)}}x(\omega)]\subseteq P_{\kappa_{x(\omega)}}x(\omega)$.

For (b), if $\kappa$ is a successor, say $\kappa=\mu^+$ then $\{x\in P_\kappa\lambda\st |x\cap\kappa|=\mu\}$ is in $\NSS_{\kappa,\lambda}^*$ and is a subset of $S$.

For (c), fix $A\subseteq\lambda$ and let $M$ be a $(\kappa,\lambda)$-model\marginpar{\tiny Change terminology.} with $A,S\in M$. Since $P_\kappa\lambda$ is weakly compact there is a $j:M\to N$ with critical point $\kappa$ such that $j(\kappa)>\lambda$ and $j"\lambda\in N$. In $N$ we have $|\kappa|<|j"\lambda|$, and hence $j"\lambda\in j((P_\kappa\lambda)\setminus S)$. Thus $S\in \NWC_{\kappa,\lambda}$.
\end{proof}

\begin{corollary}
If $P_\kappa\lambda$ is weakly compact then $\NSS_{\kappa,\lambda}\subsetneq \NWC_{\kappa,\lambda}$.
\end{corollary}

\begin{proof}
If $P_\kappa\lambda$ is weakly compact then $\NSS_{\kappa,\lambda}\subseteq\NWC_{\kappa,\lambda}$ by  Corollary \ref{corollary_NSS_contained_in_NWC} and $S=\{x\in P_\kappa\lambda\st |x\cap\kappa|=|x|\}\in\NWC_{\kappa,\lambda}\setminus\NSS_{\kappa,\lambda}$.
\end{proof}

\section{Indescribability of subsets of $P_\kappa\lambda$}\label{section_indescribability}

According to \cite{MR1635559} and \cite{MR808767}, in a set of handwritten notes, Baumgartner \cite{Baum} defined a notion of indescribability for subsets of $P_\kappa\lambda$ as follows. Give a regular cardinal $\kappa$ and a set of ordinals $A$ with $\kappa\leq|A|$, consider the hierarchy:
\begin{align*}
V_0(\kappa,A)&=A\\
V_{\alpha+1}(\kappa,A)&=P_\kappa(V_\alpha(\kappa,A))\cup V_{\alpha}(\kappa,A)\\
V_\alpha(\kappa,A)&=\bigcup_{\beta<\alpha}V_\beta(\kappa,A) \text{ for $\alpha$ a limit}
\end{align*}
Clearly $V_\kappa\subseteq V_\kappa(\kappa,A)$ and if $A$ is transitive then so is $V_\alpha(\kappa,A)$ for all $\alpha\leq\kappa$. See \cite[Section 4]{MR808767} for a discussion of the restricted axioms of $\ZFC$ satisfied by $V_\kappa(\kappa,\lambda)$ when $\kappa$ is inaccessible.

\begin{definition}[Baumgartner \cite{Baum}]\label{definition_indescribable}
Suppose $\kappa$ is a regular cardinal and $A$ is a set of ordinals with $\kappa\leq|A|$. Let $S\subseteq P_\kappa A$. We say that $S$ is \emph{$\Pi^1_n$-indescribable in $P_\kappa A$} if whenever $(V_\kappa(\kappa,A),\in,R_1,\ldots,R_k)\models\varphi$ where $k<\omega$, $R_1,\ldots,R_k\subseteq V_\kappa(\kappa,A)$ and $\varphi$ is a $\Pi^1_n$ sentence, there is an $x\in S$ such that
\[\text{$x\cap\kappa=\kappa_x$ and $(V_{\kappa_x}(\kappa_x,x),\in, R\cap V_{\kappa_x}(\kappa_x,x))\models\varphi$.}\] 
\marginpar{\tiny We need to mention that it would be equivalent to use finitely many predicates.}
\marginpar{\tiny For the weakly compact reflection principle section we need to define the $\Pi^1_1$-indescribable ideal on $P_{\kappa_x}x$ for $x\in P_\kappa\lambda$. Then, we need Carr's Theorem for this notion; that is, we want $X\subseteq P_{\kappa_x}x$ is $\Pi^1_1$-indescribable iff $\NUP_{\kappa_x,x,X}$ iff $X$ is weakly compact.}
\end{definition}

Abe \cite[Lemma 4.1]{MR1635559} showed that if $P_\kappa A$ is $\Pi^1_n$-indescribable then 
\[\Pi^1_n(\kappa,A)=\{X\subseteq P_\kappa A\st \text{$X$ is not $\Pi^1_n$-indescribable}\}\]
is a strongly normal proper ideal on $P_\kappa A$.

\begin{lemma}\cite[page 270]{MR1635559}\label{lemma_sentence}
Assuming $|A|\geq\kappa$, there is a $\Pi^1_1$-sentence $\sigma$ such that $(V_\kappa(\kappa,A),\in)\models\sigma$ if and only if $\kappa$ is inaccessible. 
\end{lemma}

\begin{lemma}\cite[Lemma 1.3]{MR1635559}\label{lemma_Abe_basic_properties_of_hierarchy}
Suppose $\kappa\leq\lambda$ are cardinals and $\kappa$ is regular.
\begin{enumerate}
\item $V_\kappa(\kappa,\lambda)=\bigcup_{x\in P_\kappa\lambda}V_{\kappa_x}(\kappa_x,x)$.
\item If $y\sqsubset x$ then $V_{\kappa_y}(\kappa_y,y)\in V_{\kappa_x}(\kappa_x,x)$.
\item If $\kappa_x= x\cap \kappa$ is inaccessible then $V_{\kappa_x}(\kappa_x,x)=\bigcup_{y\sqsubset x}V_{\kappa_y}(\kappa_y,y)$.
\item For any bijection $h:V_\kappa(\kappa,\lambda)\to P_\kappa\lambda$, $\{x\in P_\kappa\lambda\st h[V_{\kappa_x}(\kappa_x,x)]=P_{\kappa_x}x\}\in \NSS_{\kappa,\lambda}^*$.
\item If $\kappa$ is inaccessible then $\{x\in P_\kappa\lambda\st V_{\kappa_x}(\kappa_x,x)\prec V_\kappa(\kappa,\lambda)\}\in \NSS_{\kappa,\lambda}^*$.
\end{enumerate}
\end{lemma}

Abe states Lemma \ref{lemma_Abe_basic_properties_of_hierarchy} without proof; we now present a restatement and proof of Lemma \ref{lemma_Abe_basic_properties_of_hierarchy} (5) (see Lemma \ref{lemma_weak_club_of_substructures} below) since it is vital to our proof of Theorem \ref{lemma_Pi_1_0_ideal} and seems to be somewhat nontrivial.

\begin{lemma}\label{lemma_subset}
Suppose $\kappa\leq\lambda$ are cardinals and $\kappa$ is regular. If $x,y\in P_\kappa\lambda$ and $x\subseteq y$ then $V_{\beta}(\kappa_x,x)\subseteq V_{\beta}(\kappa_y,y)$ for all $\beta\leq\kappa_x$ and $V_{\kappa_x}(\kappa_x,x)\subseteq V_{\kappa_y}(\kappa_y,y)$.
\end{lemma}

\begin{proof}
Notice that $\kappa_x\leq\kappa_y$. We proceed by induction.  Clearly $V_0(\kappa_x,x)=x\subseteq y =V_0(\kappa_y,y)$. Suppose $V_\alpha(\kappa_x,x)\subseteq V_\alpha(\kappa_y,y)$, then 
\[V_{\alpha+1}(\kappa_x,x)=P_{\kappa_x}(V_\alpha(\kappa_x,x))\cup V_\alpha(\kappa_x,x)\subseteq P_{\kappa_y}(V_\alpha(\kappa_y,y))\cup V_\alpha(\kappa_y,y)=V_{\alpha+1}(\kappa_y,y).\]
Assume that $V_{\alpha}(\kappa_x,x)\subseteq V_{\alpha}(\kappa_y,y)$ for all $\alpha<\gamma\leq\kappa_x$ where $\gamma$ is a limit ordinal. Then $V_{\gamma}(\kappa_x,x)\subseteq V_\gamma(\kappa_y,y)$ follows easily by definition.
\end{proof}

\begin{lemma}\label{lemma_make_closure_point_regular}
Let $\kappa$ be an inaccessible cardinal and $X$ a set of ordinals with $\kappa\leq |X|$. Suppose $f:P_\kappa X\to P_\kappa X$ is a function such that for every $z\in P_\kappa X$ we have $\sup(z\cap\kappa)^+\subseteq f(z)$. For $x\in P_\kappa X$, if $x\cap\kappa\in\kappa$ and $f[P_{\kappa_x}x]\subseteq P_{\kappa_x}x$ then $\kappa_x=x\cap\kappa$ is a weakly inaccessible cardinal.
\end{lemma}

\begin{proof}
Suppose $x\cap\kappa$ were singular. Then some $z\in P_{\kappa_x}x$ is cofinal in $x\cap\kappa$. But then since $f[P_{\kappa_x}x]\subseteq P_{\kappa_x}x$ we have $\sup(z\cap\kappa)^+\subseteq f(z) \in P_{\kappa_x}x$, which is impossible since $\sup(z\cap\kappa)^+>\kappa_x$. This implies that $x \cap\kappa$ is a regular ordinal and thus a cardinal.

Suppose $x\cap\kappa$ were a successor cardinal, say $x\cap\kappa=\mu^+$. Since $\kappa_x=x\cap\kappa=\mu^+$ it follows that $\mu\in P_{\kappa_x}x$ and since $f[P_{\kappa_x}x]\subseteq P_{\kappa_x}x$ we conclude that $\sup(\mu)^+=\mu^+=\kappa_x\subseteq f(\mu)\in P_{\kappa_x}x$, which is impossible.
\end{proof}

\begin{lemma}\label{lemma_for_skolem}
Suppose $\kappa$ is an inaccessible cardinal and $A$ is a set of ordinals with $\kappa\leq |A|$. If $X\subseteq V_\kappa(\kappa,A)$ and $|X|<\kappa$, then there exists $x\in P_\kappa A$ such that $X\subseteq V_{\kappa_x}(\kappa_x,x)$.
\end{lemma}\marginpar{\tiny Just state the more general version of this.}

%Since $|X|<\kappa$ there is an $\alpha<\kappa$ such that $X\subseteq V_\alpha(\kappa,\lambda)$. We will prove by induction that for all $\alpha<\kappa$ if $X\subseteq V_\alpha(\kappa,\lambda)$ then there is an $x\in P_\kappa\lambda$ such that $X\subseteq V_{\kappa_x}(\kappa_x,x)$.

\begin{proof}
We will prove by induction that for all $\alpha<\kappa$, if $X\subseteq V_\alpha(\kappa, A)$ and $|X|<\kappa$ then there is an $x\in P_\kappa A$ such that $X\subseteq V_{\kappa_x}(\kappa_x,x)$. Suppose $X\subseteq V_0(\kappa, A)= A$ and $|X|<\kappa$, then $X\in P_\kappa A$. By the inaccessibility of $\kappa$ we may let $x\in P_\kappa A$ be such that $X\subseteq x$ and $|X|<\kappa_x$. Then $X\in P_{\kappa_x}x\subseteq V_{\kappa_x}(\kappa_x,x)$. Suppose $X\subseteq V_\alpha(\kappa, A)$ for some limit $\alpha<\kappa$. Let $\<\beta(i)\st i<\cf(\alpha)\>$ be cofinal in $\alpha$. For each $i$, $X\cap V_{\beta(i)}(\kappa, A)$ is a subset of $V_{\beta(i)}(\kappa, A)$ of size $<\kappa$. Thus, by induction, for each $i$ there is some $x(i)\in P_\kappa A$ such that $X\cap V_{\beta(i)}(\kappa, A)\subseteq V_{\kappa_{x(i)}}(\kappa_{x(i)},x(i))$. Let $x=\bigcup_{i<\cf(\alpha)} x(i)$. Since $x(i)\subseteq x$, by Lemma \ref{lemma_subset}, we have $V_{\kappa_{x(i)}}(\kappa_{x(i)},x(i))\subseteq V_{\kappa_x}(\kappa_x,x)$ for each $i<\cf(\alpha)$. Thus 
\[X=\bigcup_{i<\cf(\alpha)}(X\cap V_{\beta(i)}(\kappa, A))\subseteq\bigcup_{i<\cf(\alpha)} V_{\kappa_{x(i)}}(\kappa_{x(i)},x(i))\subseteq V_{\kappa_x}(\kappa_x,x).\]
Now suppose $X\subseteq V_{\alpha+1}(\kappa, A)=P_\kappa(V_\alpha(\kappa, A))\cup V_\alpha(\kappa, A)$. Then we may write $X=Y\cup Z$ for some $Y\subseteq P_\kappa(V_\alpha(\kappa, A))$ and $Z\subseteq V_\alpha(\kappa, A)$ with $Y\cap Z=\emptyset$. Let $X'=(\bigcup Y)\cup Z$, then we have $|X'|<\kappa$ and $X'\subseteq V_\alpha(\kappa, A)$. By the inductive hypothesis there is some $y\in P_\kappa A$ such that $X'\subseteq V_{\kappa_y}(\kappa_y,y)$. Now choose $x\in P_\kappa A$ with $\kappa_y\sqsubset\kappa_x$. Then $X'\subseteq V_{\kappa_y}(\kappa_y,y)\subseteq V_{\kappa_y}(\kappa_x,x)$ implies $X\subseteq X'\cup P_\kappa X'\subseteq V_{\kappa_y+1}(\kappa_x,x)\subseteq P_{\kappa_x}(\kappa_x,x)$ and thus $X\subseteq V_{\kappa_x}(\kappa_x,x)$.
\end{proof}

Here we present, with proof, a slight modification of Lemma \ref{lemma_Abe_basic_properties_of_hierarchy}(5).

\begin{lemma}\label{lemma_weak_club_of_substructures}
Suppose $\kappa$ is inaccessible and $A$ is a set of ordinals with $\kappa\leq|A|$. If $R\subseteq V_\kappa(\kappa,A)$ then
\[C=\{x\in P_\kappa A\st (V_{\kappa_x}(\kappa_x,x),\in,R\cap V_{\kappa_x}(\kappa_x,x))\prec (V_\kappa(\kappa, A),\in, R)\}\]
is in $\NSS_{\kappa, A}^*$; in other words, there is a function $f:P_\kappa A\to P_\kappa A$ such that $C_f=_{\defn}\{x\in P_\kappa A\st f[P_{\kappa_x} x]\subseteq P_{\kappa_x} x\}\subseteq C$. \marginpar{\tiny Maybe Abe assumes $\kappa$ is Mahlo? See the discussion around the definition of WNS on page 263.}
\end{lemma}

\begin{proof}

Let $\lhd$ be a wellordering of $P_\kappa A$. Define $f:P_\kappa A\to P_\kappa A$ by letting $f(z)$ be the $\lhd$-least $y\in P_\kappa A$ such that 
\begin{enumerate}
%\item $y\cap\kappa\in\kappa$,
\item $z\subsetneq y$,
\item $\sup(z\cap\kappa)^+\subseteq y$, 
\item $\Sk^{V_\kappa(\kappa, A)}(V_{\kappa_z}(\kappa_z,z),\in)\subseteq V_{\kappa_y}(\kappa_y,y)$ and
\item $|y|=y\cap\kappa=\kappa_y$. %$|z|<y\cap\kappa$
\end{enumerate}

That such a $y$ can be found follows from Lemma \ref{lemma_for_skolem} and the fact that $\{z\in P_\kappa A\st |z|=z\cap\kappa=\kappa_z\}$ is unbounded. Suppose $f[P_{\kappa_x}x]\subseteq P_{\kappa_x}x$. Since $x\cap\kappa\in\kappa$, it follows by Lemma \ref{lemma_make_closure_point_regular} that $\kappa_x=x\cap\kappa$ is a weakly inaccessible cardinal. 

Next we will argue that conditions (1) and (4) imply that $|x|=\kappa_x$. Notice that the set $D=_{\defn}\{z\in P_\kappa A\st |z|=z\cap\kappa=\kappa_z\}$ is a Jech club, meaning that $D\in\NS_{\kappa, A}^*$. Since $z\subsetneq f(z)\in D$ for every $z\in P_\kappa A$ we may apply Lemma \ref{lemma_jech_clubs_are_weak_clubs} to see that $C_f\subseteq D$ and hence $x\in D$, which implies $|x|=x\cap\kappa=\kappa_x$.

Since $|x|=x\cap\kappa=\kappa_x$ is a weakly inaccessible cardinal, we may fix a bijection $b$ from the set of successor ordinals less than $\kappa_x$ to $x$. We recursively define a $\subsetneq$-increasing sequence $\<x(i)\st i<\kappa_x\>$ in $P_{\kappa_x}x$ and an elementary chain $\<M_i\st i<\kappa_x\>$ of substructures of $(V_\kappa(\kappa, A),\in)$ as follows. Choose $x(0)\in P_{\kappa_x}x$ and let $M_0=\Sk^{V_\kappa(\kappa, A)}(V_{\kappa_{x(0)}}(\kappa_{x(0)},x(0)),\in)$. By (3) and the fact that $f[P_{\kappa_x}x]\subseteq P_{\kappa_x}x$, it follows that $M_0\subseteq V_{\kappa_y}(\kappa_y,y)$ for some $y\in P_{\kappa_x}x$ with $x(0)\subseteq y$. Given $x(i)$ and $M_i$, let $x(i+1)$ be the $\lhd$-least element of $P_{\kappa_x}x$ such that $x(i)\cup\{b(i+1)\}\subseteq x(i+1)$ and $M_i\subseteq V_{\kappa_{x(i+1)}}(\kappa_{x(i+1)},x(i+1))$. Define $M_{i+1}=\Sk^{V_\kappa(\kappa, A)}(V_{\kappa_{x(i+1)}}(\kappa_{x(i+1)},x(i+1)),\in)$. If $i<\kappa_x$ is a limit let $x(i)=\bigcup_{j<i}x(j)$ and $M_i=\bigcup_{j<i}M_j$. It follows by induction, applying the elementary chain lemma at limit stages, that for every $j<\kappa_x$ we have $i<j$ implies $M_i\prec M_j$. Thus $\<M_i\st i<\kappa_x\>$ is indeed an elementary chain of substructures of $(V_\kappa(\kappa, A),\in)$. Since $b(i+1)\in x(i+1)$ for all $i<\kappa_x$, and $b$ is a bijection from the successor ordinals less than $\kappa_x$ to $x$, it follows that 

\begin{align}
x=\bigcup_{i<\kappa_x}x(i). \label{equation_continuity}
\end{align}
Thus $\kappa_x=\bigcup_{i<\kappa_x}\kappa_{x(i)}$. Furthermore we have

\begin{align*}
\bigcup_{i<\kappa_x}M_i&=\bigcup_{i<\kappa_x}V_{\kappa_{x(i)}}(\kappa_{x(i)},x(i))\tag{by construction}\\
	&=\bigcup_{y\prec x}V_{\kappa_y}(\kappa_y,y)\tag{use  (\ref{equation_continuity}) and $\kappa_x$ inaccessible}\\
	&=V_{\kappa_x}(\kappa_x,x) \tag{Lemma \ref{lemma_Abe_basic_properties_of_hierarchy}}
\end{align*}
\noindent Since $\<M_i\st i<\kappa_x\>$ is an elementary chain, each $M_i$ is an elementary substructure of $(V_{\kappa_x}(\kappa_x,x),\in)$ and hence $(V_{\kappa_x}(\kappa_x,x),\in)\prec V_{\kappa}(\kappa, A)$ by the Tarski-Vaught test.
 
%Suppose $x\in P_\kappa A$. Since $|V_{\kappa_x}(\kappa_x,x)|<\kappa$, $\Sk^{V_\kappa(\kappa, A)}(V_{\kappa_x}(\kappa_x,x))\subseteq V_\kappa(\kappa, A)$ has size less than $\kappa$. By Lemma \ref{lemma_for_skolem}, there is a $y\in P_\kappa A$ such that $\Sk^{V_\kappa(\kappa, A)}(V_{\kappa_x}(\kappa_x,x))\subseteq V_{\kappa_y}(\kappa_y,y)$. Define $f(x)=y$.

%Suppose $x\in C_f$, in other words $f[P_{\kappa_x}x]\subseteq P_{\kappa_x}x$. We show that $V_{\kappa_x}(\kappa_x,x)\prec V_\kappa(\kappa, A)$ by verifying the Tarski-Vaught criterion. Suppose $a\in V_{\kappa_x}(\kappa_x,x)$ and $V_{\kappa}(\kappa, A)\models\exists u \varphi(u,a)$. Let $y\in P_{\kappa_x}x$ be such that $a\in V_{\kappa_y}(\kappa_y,y)$ and let $z=f(y)$ so that $z\in P_{\kappa_x}x$ and $\Sk^{V_\kappa(\kappa, A)}(V_{\kappa_y}(\kappa_y,y))\subseteq V_{\kappa_z}(\kappa_z,z)$. Then $\Sk^{V_\kappa(\kappa, A)}(V_{\kappa_y}(\kappa_y,y))\models\exists u \varphi(u,a)$. We'd like to show that $V_{\kappa_x}(\kappa_x,x)\models\exists u\varphi(u,a)$, but these structures are not transitive! Maybe we can arrange things so that the important things are happening below the transitive part of $V_{\kappa_x}(\kappa_x,x)$...

%Suppose $x\in P_\kappa A$. Since $|V_{\kappa_x}(\kappa_x,x)|<\kappa$, $\Sk^{V_\kappa(\kappa, A)}(V_{\kappa_x}(\kappa_x,x))\subseteq V_\kappa(\kappa, A)$ has size less than $\kappa$. By Lemma \ref{lemma_for_skolem}, there is a $y\in P_{\kappa} A$ such that $\Sk^{V_\kappa(\kappa, A)}(V_{\kappa_x}(\kappa_x,x))$
\end{proof}

\begin{lemma}\label{lemma_transitive_part}
If $\delta$ is inaccessible and $A$ is a set of ordinals with $\delta\leq |A|$, then $V_\delta(\delta,A)\cap V_\delta=V_\delta$. In particular, if $x\in P_\kappa\lambda$ and $\kappa_x$ is inaccessible then $V_{\kappa_x}(\kappa_x,x)\cap V_{\kappa_x}=V_{\kappa_x}$.
\end{lemma}

\begin{proof}
It suffices to show that for every $\alpha<\delta$ we have $V_\alpha(\delta,A)\cap V_\alpha=V_\alpha$, which can be done using an easy induction argument.
\end{proof}

\begin{lemma}\label{lemma_restrict_P_kappa_lambda}
For $x\in P_\kappa\lambda$ with $1\leq x\cap \kappa=\kappa_x\in \kappa$ we have $P_\kappa\lambda\cap V_{\kappa_x}(\kappa_x,x)=P_{\kappa_x}x$.
\end{lemma}

\begin{proof}
Notice that $V_1(\kappa_x,x)=(P_{\kappa_x}x)\cup x\subseteq V_{\kappa_x}(\kappa_x,x)$. Hence $P_{\kappa_x}x\subseteq P_{\kappa}\lambda\cap V_{\kappa_x}(\kappa_x,x)$. For the converse, it suffices to prove by induction that for every $\alpha<\kappa_x$ we have $P_\kappa\lambda\cap V_{\alpha}(\kappa_x,x)\subseteq P_{\kappa_x}x$. For $\alpha=0$ notice that $P_\kappa\lambda\cap V_0(\kappa_x,x)=(P_{\kappa}\lambda)\cap x= (x\cap \kappa)=\kappa_x\subseteq P_{\kappa_x}x$. Assuming that $P_\kappa\lambda\cap V_\alpha(\kappa_x,x)\subseteq P_{\kappa_x}x$, let us consider $P_\kappa\lambda\cap V_{\alpha+1}(\kappa_x,x)$. Since $P_\kappa\lambda\cap P_{\kappa_x}(V_\alpha(\kappa_x,x))\subseteq P_{\kappa_x}x$ it follows that $P_\kappa\lambda\cap V_{\alpha+1}(\kappa_x,x)=P_\kappa\lambda\cap P_{\kappa_x}(V_\alpha(\kappa_x,x))\subseteq P_{\kappa_x}x$. The limit case is trivial.
\end{proof}

As mentioned in Section \ref{section_introduction}, the next theorem suggests that one should use strong stationarity instead of stationarity when generalizing the notion of $1$-club subset of $\kappa$ to that of $P_\kappa\lambda$.

\begin{theorem}\label{lemma_Pi_1_0_ideal}
If $\kappa$ is Mahlo and $A$ is a set of ordinals with $\kappa\leq |A|$, then $S\subseteq P_\kappa A$ is in $\NSS_{\kappa, A}^+$ if and only if $S$ is $\Pi^1_0$-indescribable in $P_\kappa A$ (i.e. first-order indescribable); in other words,
\[\Pi^1_0(\kappa, A)=\NSS_{\kappa, A}.\]\marginpar{\tiny Actually we need $\Pi^1_0(\kappa_x,x)=\NSS_{\kappa_x,x}$!}
\end{theorem}

\begin{proof}
Suppose $S$ is in $\NSS_{\kappa, A}^+$, $R\subseteq V_\kappa(\kappa, A)$ and let $\varphi$ be a first order sentence with $(V_\kappa(\kappa, A),\in,R)\models\varphi$. By Lemma \ref{lemma_weak_club_of_substructures}, there is a weak club $C_f$ such that $x\in C_f$ implies $(V_{\kappa_x}(\kappa_x,x),\in, R\cap V_{\kappa_x}(\kappa_x,x))\elemsub (V_\kappa(\kappa, A),\in,R)$. If we choose $x\in C_f\cap S\neq\emptyset$ then $(V_{\kappa_x}(\kappa_x,x),\in, R\cap V_{\kappa_x}(\kappa_x,x))\models\varphi$.

Conversely, suppose $S$ is $\Pi^1_0$-indescribable, i.e. $S\in \Pi^1_0(\kappa, A)^+$, and let $C_f\subseteq P_\kappa A$ be in $\NSS_{\kappa, A}^*$ where $f:P_\kappa A\to P_\kappa A$. Since $V_1(\kappa, A)=(P_\kappa A)\cup A$, it follows that $f\subseteq V_3(\kappa, A)\subseteq V_{\kappa}(\kappa, A)$. We have 
\begin{align}
(V_\kappa(\kappa, A),\in,f,P_\kappa A)\models (\forall y\in P_\kappa A)(\exists z\in P_\kappa A) (f(y)=z).\label{label_formula}
\end{align}
Since $S$ is $\Pi^1_0$-indescribable we may fix an $x\in S$ with $x\cap\kappa=\kappa_x$ to which the formula in (\ref{label_formula}) reflects. Since $x\cap \kappa=\kappa_x$, we may apply Lemma \ref{lemma_restrict_P_kappa_lambda} to obtain $P_\kappa A\cap V_{\kappa_x}(\kappa_x,x)=P_{\kappa_x}x$. Since $f\cap V_{\kappa_x}(\kappa_x,x)$, it follows that
\[(V_{\kappa_x}(\kappa_x,x),\in,f\restrict P_{\kappa_x}x,P_{\kappa_x}x)\models (\forall y\in P_{\kappa_x}x)(\exists z\in P_{\kappa_x}x) ((f\restrict P_{\kappa_x}x)(y)=z).\] 
Therefore $x\in S\cap C_f$.
\end{proof}
\marginpar{\tiny Need to check that $P_\kappa\lambda\cap V_{\kappa_x}(\kappa_x,x)=P_{\kappa_x}x$ and that $f\cap V_{\kappa_x}(\kappa_x,x)$ makes sense by using some coding mechanism, unless $f$ is literally a subset of $V_\kappa(\kappa,\lambda)$ using the definition of ordered pairs.}

\section{Indescribability, $n$-clubs and weak compactness}\label{section_n_clubs}

First, we show that under reasonable assumptions, a set $S\subseteq P_\kappa\lambda$ is $\Pi^1_n$-indescribable if and only if $S\cap C\neq\emptyset$ for all $n$-clubs $C\subseteq P_\kappa\lambda$. Let us define the notion of $n$-club subset of $P_\kappa\lambda$. Recall that Definition \ref{definition_1_club} states $C\subseteq P_\kappa\lambda$ is \emph{$1$-club} if and only if 

\begin{enumerate}
\item $C\in \NSS_{\kappa,\lambda}^+$ and
\item $C$ is \emph{$1$-closed}, that is, for every $x\in P_\kappa\lambda$, if $\kappa_x$ is an inaccessible\footnote{Notice that we could replace ``inaccessible'' with ``Mahlo'' here to obtain an equivalent definition because $\NSS_{\kappa_x,x}^+\neq\emptyset$ and $\kappa_x$ inaccessible implies $\kappa_x$ is Mahlo, by Theorem \ref{cpl_minimal_strongly_normal_ideal}.} cardinal and $C\cap P_{\kappa_x}x\in\NSS_{\kappa_x,x}^+$ then $x\in C$.
\end{enumerate}
We generalize Hellsten's \cite[Section 2.4]{MR2026390} notion of $n$-club subset of a cardinal to the two-cardinal context as follows.

\begin{definition}
Suppose $\kappa$ is an inaccessible cardinal and $A$ is a set of ordinals with $\kappa\leq |A|$. A set $C\subseteq P_\kappa A$ is \emph{$0$-club} in $P_\kappa A$ if and only if it is a weak club in $P_\kappa A$. For $n<\omega$, we say that $C\subseteq P_\kappa A$ is \emph{$(n+1)$-club in $P_\kappa A$} if and only if $C\in \Pi^1_n(\kappa, A)^+$ and whenever $x\in P_\kappa A$ is such that $C\cap P_{\kappa_x}x\in \Pi^1_n(\kappa_x,x)^+$ and $\kappa_x$ is inaccessible, then we have $x\in C$.
\end{definition}

\begin{lemma}\label{lemma_weak_clubs_contain_1_clubs}
Suppose $\kappa\leq \lambda$ are cardinals where $\kappa$ is Mahlo and $n<\omega$. If $P_\kappa\lambda$ is $\Pi^1_n$-indescribable so that $\Pi^1_n(\kappa,\lambda)$ is a proper ideal, then every $n$-club subset of $P_\kappa\lambda$ is $(n+1)$-club.\end{lemma}

\begin{proof}
For $n=0$ we must show that if $C\subseteq P_\kappa\lambda$ is a weak club, then it is a $1$-club. Suppose $f:P_\kappa\lambda\to P_\kappa\lambda$ is a function with $C=C_f=\{x\in P_\kappa\lambda\st f[P_{\kappa_x}x]\subseteq P_{\kappa_x}x\}$. Clearly $C_f\in \NSS_{\kappa,\lambda}^+$. Suppose $x\in P_\kappa\lambda$ is such that $C_f\cap P_{\kappa_x}x\in \NSS_{\kappa_x,x}^+$ and $\kappa_x$ is inaccessible. Suppose $x\notin C_f$. Then there is some $y\in P_{\kappa_x}x$ such that $f(y)\notin P_{\kappa_x}x$. Since $\hat{y}=_{\defn}\{z\in P_{\kappa_x}x\st y\in P_{\kappa_z}z\}\in \NS_{\kappa_x,x}^*$, it follows from Lemma \ref{lemma_jech_clubs_are_weak_clubs}, that there is a $g:P_{\kappa_x}x\to P_{\kappa_x}x$ such that $C_g\subseteq\hat{y}$. Since $C_f\cap P_{\kappa_x}x$ has nontrivial intersection with every weak club subset of $P_{\kappa_x}x$, we conclude that there is some $w\in (C_f\cap P_{\kappa_x}x)\cap C_g$. Since $w\in C_g\subseteq\hat{y}$ we have $y\in P_{\kappa_w}w$, but since $f(y)\notin P_{\kappa_x}x$ we also have $f(y)\notin P_{\kappa_w}w$ and thus $w\notin C_f$.

Suppose $n>0$ and $C\subseteq P_\kappa\lambda$ is an $n$-club. Since $\Pi^1_n(\kappa,\lambda)$ is a proper ideal, we see that $C\in \Pi^1_n(\kappa,\lambda)^+$. To show that $C$ is an $(n+1)$-club, suppose $x\in P_\kappa\lambda$ is such that $C\cap P_{\kappa_x}x\in \Pi^1_n(\kappa_x,x)^+$ and $\kappa_x$ is inaccessible. Then $C\cap P_{\kappa_x}x\in \Pi^1_{n-1}(\kappa_x,x)^+$, and hence $x\in C$ since $C$ is $n$-club.
\end{proof}

Let us consider the following generalization of a standard fact (see \cite[Corollary 6.9]{Kanamori:Book}).

\begin{lemma}\label{lemma_sentence_ind}
For every $n<\omega$ there is a $\Pi^1_{n+1}$ sentence $\varphi_n$ such that for any inaccessible cardinal $\kappa$ and any set of ordinals $A$ with $\kappa\leq|A|$ we have $S$ is $\Pi^1_n$-indescribable in $P_\kappa A$ if and only if $(V_\kappa(\kappa,A),\in,S,P_\kappa A)\models\varphi_n$.
\end{lemma}

\begin{proof}
For $n=0$, since the $\Pi^1_0$-indescribability of $S$ is equivalent to its strong stationarity (by Theorem \ref{lemma_Pi_1_0_ideal}), we let $\varphi_0$ be the natural $\Pi^1_1$ description of the strong stationarity of $S$; that is, $\varphi_0$ is the $\Pi^1_1$ statement ``for all functions $f:P_\kappa A\to P_\kappa A$ there is an $x\in S$ such that $f[P_{\kappa_x}x]\subseteq P_{\kappa_x}x$''.

Suppose $n>0$. As noted in \cite[Section 4]{MR1635559}, there is a \emph{universal} $\Pi^1_n$ formula $\psi_{1,n}(X,Y)$ where $X$ a second-order variable and $Y$ a first-order variable, in the sense that for any $\Pi^1_n$ formula $\varphi(X)$ there is a $k<\omega$ such that whenever $\delta$ is an inaccessible cardinal, $A$ is a set of ordinals with $|A|\geq\kappa$ and $R\subseteq V_\delta(\delta,A)$ we have
\[(V_\delta(\delta,A),\in)\models\varphi[R]\text{ if and only if } (V_\delta(\delta,A),\in)\models\psi_{1,n}[R,k].\]
Since $\psi_{1,n}$ is $\Pi^1_n$, it follows that the statement
\[\forall X\forall Y(\psi_{1,n}[X,Y]\rightarrow\exists x\in S ((V_{\kappa_x}(\kappa_x,x),\in)\models\psi_{1,n}[X\cap V_{\kappa_x}(\kappa_x,x), Y])),\]
which we denote by $\varphi_n$, is $\Pi^1_{n+1}$. It is straightforward to see that $\varphi_n$ satisfies the conclusion of the lemma.
\end{proof}

Generalizing \cite[Theorem 2.4.2]{MR2026390} we obtain the following result which was mentioned in Section \ref{section_introduction}.

\begin{theorem12}
Suppose $\kappa\leq\lambda$ are cardinals with $\lambda^{<\kappa}=\lambda$, $n<\omega$ and $P_\kappa\lambda$ is $\Pi^1_n$-indescribable. Then $S\subseteq P_\kappa\lambda$ is $\Pi^1_n$-indescribable if and only if for all $n$-clubs $C\subseteq P_\kappa\lambda$ we have $S\cap C\neq\emptyset$.
\end{theorem12}

\begin{proof}
The cases $n=0$ and $n=1$ follow directly from Theorem \ref{lemma_Pi_1_0_ideal}. 

Suppose $1\leq n<\omega$. For the forward direction, suppose $S\subseteq P_\kappa\lambda$ is $\Pi^1_n$-indescribable and suppose $C\subseteq P_\kappa\lambda$ is an $n$-club. Since $C\in \Pi^1_{n-1}(\kappa,\lambda)^+$, it follows that $(V_\kappa(\kappa,\lambda),\in,C,P_\kappa\lambda)\models\varphi_{n-1}\land \sigma$ where $\varphi_{n-1}$ is the $\Pi^1_n$ sentence from Lemma \ref{lemma_sentence_ind} and $\sigma$ is the $\Pi^1_1$ sentence from Lemma \ref{lemma_sentence} asserting that $\kappa$ is inaccessible. Since $S\in \Pi^1_n(\kappa,\lambda)^+$ there is an $x\in S$ such that $x\cap \kappa=\kappa_x$ and
\[(V_{\kappa_x}(\kappa_x,x),\in,C\cap P_{\kappa_x}x,P_{\kappa_x}x)\models\varphi_{n-1}\land\sigma.\]
This implies $C\cap P_{\kappa_x} x$ is in $\Pi^1_{n-1}(\kappa_x,x)^+$ and $\kappa_x$ is inaccessible. Thus $x\in C$ since $C$ is $n$-club.

Conversely, suppose $S$ intersects every $n$-club subset of $P_\kappa\lambda$. Let $R\subseteq V_\kappa(\kappa,\lambda)$ and let $\varphi=\forall X\psi(X)$ be a $\Pi^1_n$ sentence where $\psi(X)$ is a $\Sigma^1_{n-1}$ formula such that $(V_\kappa(\kappa,\lambda),\in,R)\models \varphi$. It suffices to show that 
\[D=\{x\in P_\kappa\lambda\st (V_{\kappa_x}(\kappa_x,x),\in,R\cap V_{\kappa_x}(\kappa_x,x))\models\varphi\}\]
is $n$-club. 

First let us show that $D\in\Pi^1_{n-1}(\kappa,\lambda)^+$. Suppose not, then
\[E=P_\kappa\lambda\setminus D=\{x\in P_\kappa\lambda\st(V_{\kappa_x}(\kappa_x,x),\in,R\cap V_{\kappa_x}(\kappa_x,x))\models\exists X\lnot\psi(X)\}\]
is in $\Pi^1_{n-1}(\kappa,\lambda)^*$. By our inductive hypothesis, this implies that $E$ contains an $(n-1)$-club subset of $P_\kappa\lambda$, and since $(n-1)$-clubs are $n$-clubs we see that $E$ contains an $n$-club. Since $\Pi^1_n(\kappa,\lambda)$ is a proper ideal, $E$ is therefore $\Pi^1_n$-indescribable. Thus, from our assumption that $(V_\kappa(\kappa,\lambda),\in,R)\models\forall X\psi(X)$, we may conclude that there is an $x\in E$ such that $(V_{\kappa_x}(\kappa_x,x),\in,R\cap V_{\kappa_x}(\kappa_x,x))\models\forall X\psi(X)$, a contradiction.

Next, we show that for every $x\in P_\kappa\lambda$, if $D\cap P_{\kappa_x}x\in\Pi^1_{n-1}(\kappa_x,x)^+$ then $x\in D$. Suppose $D\cap P_{\kappa_x}x\in\Pi^1_{n-1}(\kappa_x,x)^+$ but $x\notin D$. Then $(V_{\kappa_x}(\kappa_x,x),\in,R\cap V_{\kappa_x}(\kappa_x,x))\models\lnot\psi[A]$ for some $A\subseteq V_{\kappa_x}(\kappa_x,x)$. Since $\lnot\psi[A]$ is a $\Pi^1_{n-1}$ sentence, it follows that there is some $y\in D\cap P_{\kappa_x}x$ such that
\[(V_{\kappa_y}(\kappa_y,y),\in,R\cap V_{\kappa_y}(\kappa_y,y),A\cap V_{\kappa_y}(\kappa_y,y))\models\lnot\psi[A],\]
a contradiction. 
\end{proof}

From Theorem \ref{theorem_n_club} we can easily show that $n$-clubs are measure one with respect to any supercompactness ultrafilter.

\begin{corollary}\label{corollary_1_clubs_are_normal_measure_one}
Suppose $\kappa\leq\lambda$ are cardinals with $\lambda^{<\kappa}=\lambda$ and $\kappa$ is $\lambda$-supercompact. If $U$ is a normal fine $\kappa$-complete nonprincipal ultrafilter on $P_\kappa\lambda$ then for all $n<\omega$ if $C\subseteq P_\kappa\lambda$ is $n$-club then $C\in U$.
\end{corollary}

\begin{proof}
Let $j:V\to N$ be the ultrapower by $U$. So the critical point of $j$ is $\kappa$, $j(\kappa)>\lambda$ and $j"\lambda\in N$. It is easy to see that every $0$-club is in $U$.

%Suppose $n=1$ and let $C$ be $1$-club in $P_\kappa\lambda$. By elementarity $j(C)$ is a $1$-club subset of $j(P_\kappa\lambda)$ in $N$. It will suffice to show that the set $j(C)\cap P_{j(\kappa)_{j"\lambda}}j"\lambda=j(C)\cap P_\kappa j"\lambda$ is strongly stationary as a subset of $P_\kappa j"\lambda$ in $N$, because this implies $j"\lambda\in j(C)$. That is, we must show that $j(C)\cap P_\kappa j"\lambda\in(\NSS_{\kappa,j"\lambda}^+)^N$.\marginpar{\tiny Here we are using $\NSS_{\kappa,X}$.} In $N$, fix a function $f:P_\kappa j"\lambda \to P_\kappa j"\lambda$ and let $C_f=\{y\in P_\kappa j"\lambda\st f[P_{\kappa_y}y]\subseteq P_{\kappa_y}y\}$. Let us argue that the preimage $j^{-1}[C_f]=\{x\in P_\kappa\lambda\st j(x)\in C_f\}$ is a weak club subset of $P_\kappa\lambda$. Define $F:P_\kappa\lambda\to P_\kappa\lambda$ by $F(x)= j^{-1}\circ f\circ j(x)$. Since $|x|<\crit(j)$ we have $j(x)=j"x$. Let us show that $C_F=j^{-1}[C_f]$. This follows from the fact that $x\in C_F$ $\iff$ $F[P_{\kappa_x}x]\subseteq P_{\kappa_x}x$ $\iff$ $f[P_{j(\kappa)_{j(x)}}j(x)]\subseteq P_{j(\kappa)_{j(x)}}j(x)$ $\iff$ $j(x)\in C_f$.

For $n>0$, suppose $C$ is $n$-club in $P_\kappa\lambda$. Then by elementarity, in $N$, $j(C)$ is $n$-club in $j(P_\kappa\lambda)$. It will suffice to show that the set $j(C)\cap P_{j(\kappa)_{j"\lambda}}j"\lambda=j(C)\cap P_\kappa j"\lambda$ is a $\Pi^1_{n-1}$-indescribable subset of $P_\kappa j"\lambda$ in $N$, because this implies $j"\lambda\in j(C)$. By Theorem \ref{theorem_n_club}, it will suffice to show that, in $N$, $j(C)\cap P_\kappa j"\lambda$ intersects every $(n-1)$-club subset of $P_\kappa j"\lambda$. In $N$, fix an $(n-1)$-club $D\subseteq P_\kappa j"\lambda$. Since $j\restrict P_\kappa\lambda\to P_\kappa j"\lambda$ is a bijection and $j$ is a supercompactness ultrapower, it follows that $j^{-1}[D]$ is an $(n-1)$-club subset of $P_\kappa\lambda$. Since $C$ is an $n$-club subset of $P_\kappa\lambda$ in $V$, it follows that $C$ is $\Pi^1_{n-1}$-indescribable, and thus there is some $x\in C\cap j^{-1}[D]$. Then we have $j(x)=j"x\in j(C)\cap D\cap P_\kappa j"\lambda$. Thus, in $N$, $j(C)\cap P_\kappa j"\lambda$ is $\Pi^1_{n-1}$-indescribable and so $j"\lambda\in j(C)$.
\end{proof}

\begin{corollary}\label{corollary_normal_measure_one_sets_are_weakly_compact}
Suppose that $\kappa$ is $\lambda$-supercompact where $\kappa\leq\lambda$ and $\lambda^{<\kappa}=\lambda$. If $U$ is a normal fine $\kappa$-complete nonprincipal ultrafilter on $P_\kappa\lambda$ then for all $n<\omega$ we have
\[\Pi^1_n(\kappa,\lambda)^*\subseteq U\subseteq \Pi^1_n(\kappa,\lambda)^+.\]
\end{corollary}

Next, let us establish Theorem \ref{theorem_characterization_embedding}, which we restate here for the reader's convenience.

\begin{theorem14}
Suppose $\kappa\leq\lambda$ are cardinals such that $\kappa$ is inaccessible and $\lambda^{<\kappa}=\lambda$. Then a set $W\subseteq P_\kappa\lambda$ is $\Pi^1_1$-indescribable if and only if it is weakly compact.
\end{theorem14}

To do this we will use the filter characterization of weakly compact subsets of $P_\kappa\lambda$ due to Schanker \cite{MR2989393} given above in Lemma \ref{lemma_basic_embedding_characterizations}(3), which strongly resembles a filter characterization of $\Pi^1_1$-indescribable sets due to Carr \cite{MR808767}.

\begin{definition}[\cite{MR808767}]
The \emph{normal ultrafilter property for $X\in \NS_{\kappa,\lambda}^+$}, written $\NUP_{\kappa,\lambda,X}$ states that for any $\kappa$-complete field $B$ of subsets of $P_\kappa\lambda$ such that $|B|=\lambda$, $X\in B$ and $(\forall\alpha<\lambda)(\widetilde{\{\alpha\}}=_{\defn}\{x\in P_\kappa\lambda\st\alpha\in x\}\in B)$, and for any collection $G=\{g_\alpha\st\alpha<\lambda\}$ of regressive functions on $P_\kappa\lambda$ such that $(\forall\alpha<\lambda)(\forall\beta<\lambda)(g_\alpha^{-1}(\{\beta\})\in B)$, there is a $\kappa$-complete ultrafilter $U$ in $B$ such that $X\in U$, $(\forall\alpha<\lambda)(\widetilde{\alpha}\in U)$ and every function in $G$ is constant on a set in $U$.
\end{definition}

For the reader's convenience, let us recall that Carr showed that, under reasonable assumptions, $\NUP_{\kappa,\lambda,X}$ is equivalent to $X$ being $\Pi^1_1$-indescribable by using the following generalization of a characterization of weakly compact cardinals due to Shelah \cite{MR550384}.

\begin{definition}[\cite{MR808767}]
We say that $X\subseteq P_\kappa\lambda$ has the \emph{$\lambda$-Shelah property} if and only if for every sequence of functions $\<f_x\st x\in X\>\in\prod\{^x 2\st x\in X\}$ 
\[(\exists f:\lambda\to \lambda)(\forall x\in P_\kappa\lambda) (\exists y\in X\cap\widetilde{x})(f_y\restrict x=f\restrict x)\]
where $\widetilde{x}=\{z\in P_\kappa\lambda\st x\subseteq y\}$.
\end{definition}

\begin{theorem}[Theorem 3.5 and Theorem 4.7 in \cite{MR808767}]\label{theorem_carr}
Suppose $\kappa\leq\lambda$ are cardinals, $\kappa$ is inaccessible and $\lambda^{<\kappa}=\lambda$. For every set $X\subseteq P_\kappa\lambda$ we have
\[
\text{$X$ is $\Pi^1_1$-indescribable}\iff\text{$X$ is $\lambda$-Shelah}\\
	\iff\text{$\NUP_{\kappa,\lambda,X}$}.\]
\end{theorem}

\begin{proof}[Proof of Theorem \ref{theorem_characterization_embedding}] By Theorem \ref{theorem_carr} and Theorem \ref{lemma_basic_embedding_characterizations}, it suffices to show that Carr's filter property $\NUP_{\kappa,\lambda,W}$ is equivalent to Schanker's filter property Theorem \ref{lemma_basic_embedding_characterizations}(3), but it is easy to see that Schanker's filter property is a slight reformulation of Carr's filter property. Notice that if $\mathcal{A}$ is a collection of subsets of $P_\kappa\lambda$ is in Theorem \ref{lemma_basic_embedding_characterizations}(3), then by the inaccessibility of $\kappa$, there is a $\kappa$-complete collection $B$ of subsets of $P_\kappa\lambda$. Also notice that in Carr's statement of $\NUP_{\kappa,\lambda,W}$, the assertion that $U$ is a $\kappa$-complete ultrafilter in $U$ means the same thing as Schanker's statement that $U$ is a $\kappa$-complete filter \emph{measuring all sets in $B$}.
\end{proof}

\begin{corollary}\label{corollary_filter_base}
Suppose $\kappa\leq\lambda$ are cardinals and $\lambda^{<\kappa}=\lambda$. If $P_\kappa\lambda$ is weakly compact then the following hold.
\begin{align*}
\NWC_{\kappa,\lambda}&=\Pi^1_1(\kappa,\lambda)=\{Z\subseteq P_\kappa\lambda\st \text{$Z\cap C=\emptyset$ for some $1$-club $C\subseteq P_\kappa\lambda$}\}\\
\NWC_{\kappa,\lambda}^+&=\Pi^1_1(\kappa,\lambda)^+=\{W\subseteq P_\kappa\lambda\st \text{$W\cap C\neq\emptyset$ for every $1$-club $C\subseteq P_\kappa\lambda$}\}\\
\NWC_{\kappa,\lambda}^*&=\Pi^1_1(\kappa,\lambda)^*=\{C\subseteq P_\kappa\lambda\st \text{$C$ contains a $1$-club}\}
\end{align*}
\end{corollary}

\section{Applications}\label{section_applications}

\subsection{On a question of Cox-L\"ucke}\label{section_cox_lucke}

The following question was posed in Cox-L\"ucke. See Section \ref{section_introduction} for relevant background and definitions.
\begin{question}\cite[Question 7.4]{MR3620068}\label{question_cox_lucke}
Assume $(\kappa^+)^{<\kappa}=\kappa^+$. Let $\kappa$ be an inaccessible cardinal such that there is a normal filter $\mathcal{F}$ on $P_\kappa\kappa^+$ with the property that every partial order of cardinality $\kappa^+$ that satisfies the $\kappa$-chain condition is $\mathcal{F}$-layered. Must $\kappa$ be measurable?
\end{question}

The answer is no. The proof of the following lemma is very similar to that of \cite[Lemma 4.3]{MR3620068}.

\begin{lemma}\label{lemma_cox_lucke}
Suppose $P_\kappa\lambda$ is weakly compact. Then every partial order of cardinality at most $\lambda$ that satisfies the $\kappa$-chain condition is $\NWC_{\kappa,\lambda}^*$-layered.
\end{lemma}

\newcommand{\Reg}{\mathop{\rm Reg}}

\begin{proof}
Fix a surjection $s:\lambda\to \P$. We must show that $X=\{x\in P_\kappa\lambda\st s[x]\in\Reg_\kappa(\P)\}\in\NWC_{\kappa,\lambda}^*$. Let $M\models\ZFC^-$ be transitive of size $\lambda^{<\kappa}$ with $\lambda,X,\P,s,\Reg_\kappa(\P),\ldots\in M$ and let $j:M\to N$ be an elementary embedding with critical point $\kappa$ such that $j(\kappa)>\lambda$ and $j"\lambda\in N$.

Just as in Cox-L\"ucke, $j[\P]$ is a suborder of $j(\P)$ and $j\restrict\P:\P\to j[\P]$ is an isomorphism of partial orders. If $A$ is a maximal antichain of $j[\P]$ then $j^{-1}[A]$ is a maximal antichain of $\P$ and hence $|A|<\kappa$. Since $\crit(j)=\kappa$, it follows by elementarity that $A=j[j^{-1}[A]]=j(j^{-1}[A])$ is a maximal antichain of $j(\P)$. Hence, in $N$, $j[\P]$ is a regular suborder of $j(\P)$. Since $j(s)[j"\lambda]=j[s[\lambda]]=j[\P]\in\Reg_{j(\kappa)}(j(\P))^N=j(\Reg_\kappa(\P))$ we conclude that $j"\lambda\in j(X)$. Thus $X\in\NWC_{\kappa,\lambda}^*$.
\end{proof}

Recall that, as a matter of terminology, $P_\kappa\lambda$ is weakly compact if and only if $\kappa$ is nearly $\lambda$-supercompact. Schanker proved that if the near $\kappa^+$-supercompactness of $\kappa$ is indestructible by the forcing to add $\kappa^+$ Cohen subsets of $\kappa$, then the near $\kappa^+$-supercompactness of $\kappa$ is indestructible by the forcing to add any number of Cohen subsets of $\kappa$. This allowed Schanker to then show \cite[Theorem 4.10 (2)]{MR2989393} that if $\kappa$ is nearly $\kappa^+$-supercompact and $2^\kappa=\kappa^+$ then there is a forcing extension $V[G]$ in which $\kappa$ is nearly $\kappa^+$-supercompact and the $\GCH$ fails first at $\kappa$, hence $\kappa$ is not measurable. Translating Schanker's results into our terminology we obtain the following.

\begin{proposition}
Suppose $P_\kappa\kappa^+$ is weakly compact and $\GCH$ holds. There is a cofinality-preserving forcing $\P$ such that if $G\subseteq\P$ is generic then in $V[G]$ the following hold.
\begin{enumerate}
\item $\GCH$ fails first at $\kappa$, hence $(\kappa^+)^{<\kappa}=\kappa^+$ and $\kappa$ is not measurable.
\item $(P_\kappa\kappa^+)^{V[G]}$ is weakly compact, hence $\mathcal{F}=_{\defn}(\NWC_{\kappa,\kappa^+}^*)^{V[G]}$ is a nontrivial normal ideal and by Lemma \ref{lemma_cox_lucke}, every partial order of cardinality at most $\kappa^+$ that satisfies the $\kappa$-chain condition is $\mathcal{F}$-layered.
\end{enumerate}
\end{proposition}

This answers Question \ref{question_cox_lucke} and establishes Theorem \ref{theorem_answer} in the case that $\lambda=\kappa^+$. The case in which $\lambda\geq\kappa^{++}$ requires more work: the usual reflection arguments \cite[Theorem 4.10 (3)]{MR2989393} show if $\kappa$ is nearly $\kappa^{++}$-supercompact then $\GCH$ cannot fail first at $\kappa$. However, by carrying out a delicate argument using the \emph{lottery preparation} \cite{MR1736060} and the fact that in forcing extensions $V\subseteq V[G]$ satisfying the \emph{$\delta$-approximation and cover properties} for some $\delta<\kappa$, definable elementary embeddings $h:V[G]\to N$ with critical point $\kappa$ must lift ground model embeddings \cite[Corollary 8]{MR2063629}, Schanker proved the following.
\begin{theorem}[Schanker \cite{MR2989393}]
If $\kappa$ is nearly $\lambda$-supercompact for some $\lambda\geq 2^\kappa$ such that $\lambda^{<\lambda}=\lambda$, then there exists a forcing extension preserving all cardinals and cofinalities above $\kappa$ where $\kappa$ is nearly $\lambda$-supercompact but not measurable. Furthermore, in this extension $2^\kappa=\lambda^+$, and if the $\SCH$ hold below $\kappa$ in the ground model, then no cardinals or cofinalities were collapsed.
\end{theorem}
Again, by translating the previous theorem of Schanker's to our terminology, and applying Lemma \ref{lemma_cox_lucke} we obtain Theorem \ref{theorem_answer}.

%The Koszmider poset $\K_{\kappa,\lambda}$ is the collection of all $p\subseteq P_\kappa\lambda$ which satisfy the conditions (1)--(5) in the definition of $(\kappa,\lambda)$-semimorass and which have a $\subseteq$-largest element, denoted by $x^p$. We also demand that if $\rk(x^p)>0$ then $x^p=\bigcup (p\restrict x^p)$. The ordering is defined by $p\leq q$ if and only if $x^q\in p$ and $q=(p\restrict x^q)\cup\{x^q\}$. Koszmider \cite{MR1320107} proved that if $\kappa^{<\kappa}=\kappa$ then $\K_{\kappa,\lambda}$ is a ${<}\kappa$-directed closed $\kappa^+$-c.c. poset and if $G\subseteq \K_{\kappa,\lambda}$ is a generic filter over $V$ then $\mu=\bigcup G$ is a stationary $(\kappa,\lambda)$-semimorass. 

%We will argue that if $P_\kappa\lambda$ is weakly compact, forcing with a modified version of Koszmider's poset, call it $\mathbb{C}_{\kappa,\lambda}$, adds a $1$-club $(\kappa,\lambda)$-semimorass $\mu\subseteq (P_\kappa\lambda)^{V^{\mathbb{C}_{\kappa,\lambda}}}$ and forces $(P_\kappa\lambda)^{\mathbb{C}_{\kappa,\lambda}}$ to be weakly compact.

%Koszmider also proved that if $\mu\subseteq P_\kappa\lambda$ is a $(\kappa,\lambda)$-semimorass then $\mu$ is not a club subset of $P_\kappa\lambda$. 

\subsection{Two-cardinal indescribability and generic embeddings}\label{section_generic_embeddings}

Let us recall the following standard fact concerning generic ultrapowers (see \cite{MR2768692} for more information).

\begin{lemma}[Folklore]\label{lemma_normal_generic_ultrapower}
Suppose $\kappa$ is regular, $\kappa\leq\lambda$ with $\lambda^{<\kappa}=\lambda$, and $I$ is a $\kappa$-complete normal fine ideal on $P_\kappa\lambda$. If $G\subseteq P(P_\kappa\lambda)/I$ is generic and $j:V\to M=V^{P_\kappa\lambda}/G\subseteq V[G]$ is the corresponding generic ultrapower then the following conditions hold.
\begin{enumerate}
\item $G$ extends the filter $I^*$ dual to $I$.
\item $\crit(j)=\kappa$ and $j(\kappa)>\lambda$.
\item $[\id]_G=j"\lambda\in M$ and thus for all $X\in P(P_\kappa\lambda)^V$ we have $X\in G$ if and only if $j"\lambda\in^M j(X)$.
\item For every function $f:P_\kappa\lambda\to V$ in $V$ we have $j(f)(j"\lambda)=[f]_G$.
\item $M$ is wellfounded up to $(\lambda^+)^V$.
\end{enumerate}
\end{lemma}

The previous lemma easily yields the following standard result.

\begin{proposition}[Folklore]
Suppose $\kappa$ is regular, $\kappa\leq\lambda$ and $\lambda^{<\kappa}=\lambda$. A set $S\subseteq P_\kappa\lambda$ is stationary if and only if there is a generic elementary embedding $j:V\to M\subseteq V[G]$ with critical point $\kappa$ such that $j(\kappa)>\lambda$ and $j"\lambda\in j(S)\cap M$.
\end{proposition}

\begin{proof}
Suppose $S$ is stationary and let $G\subseteq P(\kappa)/(\NS_{\kappa,\lambda}\restrict S)$ be generic. Since $\NS_{\kappa,\lambda}\restrict S$ is a $\kappa$-complete normal ideal on $P_\kappa\lambda$ we have $\crit(j)=\kappa$ and $[\id]_G=j"\lambda$. Hence $j"\lambda\in M$. Since $\lambda=j(f)(j"\lambda)=[f]_G$ where $f(x)=\ot(x)$ and $j(\kappa)=[c_\kappa]_G$ we have $\lambda<j(\kappa)$. Clearly $S\in (\NS_{\kappa,\lambda}\restrict S)^*\subseteq G$ and hence $\kappa\in j(S)$.

Conversely, suppose that $j:V\to M\subseteq V[G]$ is a generic elementary embedding with critical point $\kappa$ such that $j(\kappa)>\lambda$ and $j"\lambda\in j(S)\cap M$, where $j$ is obtained by forcing with some poset $\P$. Fix a club $C\subseteq P_\kappa\lambda$ in $V$. Then
\[I=\{X\in P(P_\kappa\lambda)\st\ \forces_\P j"\lambda\notin j(X)\}\]
is a normal ideal on $P_\kappa\lambda$ in $V$. This implies that $C\in I^*$, in other words, $\forces_\P$ $j"\lambda\in j(C)$. Thus $M\models j(S)\cap j(C)\not=\emptyset$ and by elementarity $S\cap C\not=\emptyset$.
\end{proof}

The following lemma is a straightforward generalization of standard fact about $\Pi^m_n$-indescribable filter on a cardinal $\kappa$ (taking $n=m=1$ in the next lemma yields a result proven above; see the proof of Theorem \ref{theorem_n_club} in Section \ref{section_n_clubs}).

\begin{lemma}\label{lemma_in_the_filter}
Suppose $\kappa\leq\lambda$ are cardinals and $P_\kappa\lambda$ is $\Pi^m_n$-indescribable. Further suppose that $R_1,\ldots,R_k\subseteq V_\kappa(\kappa,\lambda)$ where $k<\omega$ and $\varphi$ is a $\Pi^m_n$ sentence. If $(V_\kappa(\kappa,\lambda),\in,R_1,\ldots,R_k)\models\varphi$ then the set
\[\{x\in P_\kappa\lambda\st (V_{\kappa_x}(\kappa_x,x),\in,R_1\cap V_{\kappa_x}(\kappa_x,x),\ldots,R_k\cap V_{\kappa_x}(\kappa_x,x))\models\varphi\}\]
is in the filter $\Pi^m_n(\kappa,\lambda)^*$.
\end{lemma}

Next, we will provide a characterization of $\Pi^m_n$-indescribable subsets of $P_\kappa\lambda$ in terms of generic embeddings.

\begin{proposition} Suppose $n,m<\omega$, $\kappa$ is regular, $\lambda\geq\kappa$ is a cardinal with $\lambda^{<\kappa}=\lambda$ and $S\subseteq P_\kappa\lambda$. The following are equivalent.
\begin{enumerate}
\item $S$ is $\Pi^m_n$-indescribable in $P_\kappa\lambda$
\item There is a generic embedding $j:V\to M\subseteq V[G]$ with critical point $\kappa$ such that $\crit(j)=\kappa$, $j(\kappa)>\lambda$, $j"\lambda\in j(S)\cap M$ and for every $\Pi^m_n$-sentence $\varphi$ over $(V_\kappa(\kappa,\lambda),\in,R_1,\ldots, R_k)$ where $R_1,\ldots,R_k\in P(V_\kappa(\kappa,\lambda))^V$ it follows that
\[((V_\kappa(\kappa,\lambda),\in,R_1,\ldots,R_k)\models\varphi)^V\]
implies
\[((V_\kappa(\kappa,j"\lambda),\in, j(R_1)\cap V_\kappa(\kappa,j"\lambda),\ldots,j(R_k)\cap V_\kappa(\kappa,j"\lambda))\models\varphi)^M.\]
\end{enumerate}
\end{proposition}

\begin{proof}
To see that (1) implies (2), suppose $S\subseteq P_\kappa\lambda$ is $\Pi^m_n$-indescribable. Let $G$ be generic for the poset $P(P_\kappa\lambda)/(\Pi^m_n(\kappa,\lambda)\restrict S)-\{[\emptyset]\}$. Since $G$ extends the filter $(\Pi^m_n(\kappa,\lambda)\restrict S)^*$ and $S\in(\Pi^m_n(\kappa,\lambda)\restrict S)^*$, it follows that $S\in G$. Thus, by Lemma \ref{lemma_normal_generic_ultrapower}, if we let $j:V\to M=V^{P_\kappa\lambda}/G\subseteq V[G]$ be the generic ultrapower obtained from $G$, then $j"\lambda\in j(S)\cap M$. Furthermore, if $\varphi$ and $R_1,\ldots, R_k$ are as in the statement of the proposition and \[((V_\kappa(\kappa,\lambda),\in,R_1,\ldots,R_k)\models\varphi)^V,\]
then it follows by Lemma \ref{lemma_in_the_filter} and Lemma \ref{lemma_normal_generic_ultrapower} that 
\[((V_\kappa(\kappa,j"\lambda),\in, j(R_1)\cap V_\kappa(\kappa,j"\lambda),\ldots,j(R_k)\cap V_\kappa(\kappa,j"\lambda))\models\varphi)^M.\]

Conversely, if (2) holds then it follows by elementarity that (1) holds.
\end{proof}

\subsection{Two-cardinal weakly compact diamond}\label{section_diamond}
First, as an application of the $1$-club characterization of weak compactness in Corollary \ref{corollary_filter_base}, we will show that if $\kappa$ is large enough then for every $\lambda\geq\kappa$ with $\lambda^{<\kappa}=\lambda$ and every weakly compact $W\subseteq P_\kappa\lambda$, $\wcdiamond_{\kappa,\lambda}(W)$ holds (see Definition \ref{definition_weakly_compact_diamond}).

\begin{remark}
In what follows we will identify subsets $X\subseteq \lambda$ with functions $X:\ot(X)\to \lambda$ enumerating the elements of $X$ in increasing order; in other words, $X(\alpha)$ denotes the $\alpha$-th element of $X$ where $\alpha<\ot(X)$.
\end{remark}

\begin{proposition}
Suppose $\kappa$ is supercompact and $\lambda\geq\kappa$ is a cardinal with $\lambda^{<\kappa}=\lambda$. If $W\subseteq P_\kappa\lambda$ weakly compact then $\wcdiamond_{\kappa,\lambda}(W)$ holds.
\end{proposition}

\begin{proof}

Suppose $\kappa$ is supercompact and let $\ell:\kappa\to V_\kappa$ be a Laver function \cite{MR0472529}, that is, for any $\lambda$ and any $x\in H_{\lambda^+}$, there is a $\lambda$-supercompactness embedding $j:V\to M$ with critical point $\kappa$, $j(\kappa)>\lambda$, $j"\lambda\in M$ and $j(\ell)(\kappa)=x$. 

Fix $\lambda\geq\kappa$ with $\lambda^{<\kappa}=\lambda$ and fix a weakly compact set $W\subseteq P_\kappa\lambda$. For each $z\in W$ with $z\cap\kappa\in\kappa$ and $\ell(z\cap\kappa)\subseteq\ORD$ define $a_z=\{z(\beta)\st\beta<\ot(z)\land \ell(z\cap\kappa)(\beta)=1\}$ where $\ell(z\cap\kappa)(\beta)$ is the $\beta$-th element of $\ell(z\cap\kappa)$. Otherwise let $a_z$ be arbitrary. Let us argue that the set $E_X=\{z\in P_\kappa\lambda\st X\cap z=a_z\}$ is weakly compact. 

Fix a $1$-club $C\subseteq P_\kappa\lambda$. By Corollary \ref{corollary_filter_base}, it will suffice to show that $j"\lambda\in j(C)\cap j(E_X)$ where $j$ is a $\lambda$-supercompactness embedding.\footnote{Let us emphasize that, in this context $j"\lambda\in j(E_X)$ does not directly imply $E_X$ is weakly compact because $j$ is a supercompactness embedding, but the fact that $E_X$ is weakly compact follows from Corollary \ref{corollary_filter_base} or Corollary \ref{corollary_normal_measure_one_sets_are_weakly_compact}.} Take $j:V\to M$ to be a $\lambda$-supercompactness embedding with $j(\ell)(\kappa)=E_X$. We certainly have $j"\lambda\in j(C)$. Let $j(\<a_z\st z\in P_\kappa\lambda\>)=\<\bar{a}_z\st z\in j(P_\kappa\lambda)\>$. Since $\bar{a}_{j"\lambda}=\{j"\lambda(\beta)\st \beta<\lambda\land j(\ell)(\kappa)(\beta)=1\}=j"X=j(X)\cap j"\lambda$, it follows that $j"\lambda\in j(E_X)$.
\end{proof}

When arguing that ultrapower embeddings $j:V\to M$ by normal fine $\kappa$-complete measures on $P_\kappa\lambda$ can be extended to forcing extensions by an Easton-support iteration $\P$, one often uses a function $f:\kappa\to\kappa$ satisfying $j(f)(\kappa)>\lambda$ to ensure that the tail of $j(\P)$ will be sufficiently closed. The same is true when lifting elementary embeddings witnessing the weak compactness of subsets of $P_\kappa\lambda$.

\begin{definition}\label{definition_menas_property}
Suppose $P_\kappa\lambda$ is weakly compact. We say that a function $f:\kappa\to\kappa$ has the \emph{Menas property for weakly compact subsets of $P_\kappa\lambda$} if and only if for every weakly compact $W\subseteq P_\kappa\lambda$ and every $A\subseteq\lambda$ there is a transitive $M\models\ZFC^-$ of size $\lambda^{<\kappa}$ with $\lambda,A,W,f\in M$, a transitive $N$ and an elementary embedding $j:M\to N$ with critical point $\kappa$ such that $j(\kappa)>\lambda$, $j"\lambda\in j(W)$ and $j(f)(\kappa)>\lambda$.
\end{definition}

The proof of the following lemma is essentially the same as that of \cite[Lemma 3.3]{MR2989393}

\begin{lemma}
Suppose $P_\kappa\lambda$ is weakly compact. Then there is a function $f:\kappa\to\kappa$ with the Menas property for weakly compact subsets of $P_\kappa\lambda$.
\end{lemma}

\begin{theorem}
Suppose $W\subseteq P_\kappa\lambda$ is weakly compact and $\GCH$ holds. There is a cofinality-preserving forcing extension $V[G]$ in which $W$ is a weakly compact subset of $(P_\kappa\lambda)^{V[G]}$ and $\wcdiamond_{\kappa,\lambda}(W)$ holds.
%\[V[G]\models(\text{$W$ is a weakly compact subset of $P_\kappa\lambda$}) \land \wcdiamond_{\kappa,\lambda}(W).\] %in which $W$ remains a weakly compact subset of $(P_\kappa\lambda)^{V[G]}$ and $\wcdiamond_{\kappa,\lambda}(W)$ holds.
\end{theorem}

\begin{proof}
Let $f:\kappa\to\kappa$ be a function with the Menas property for weakly compact subsets of $P_\kappa\lambda$. Let $\P_{\kappa+1}=\<\P_\alpha,\dot{\Q}_\beta\st\alpha\leq\kappa+1,\beta\leq\kappa\>$ be the Easton-support iteration such that if $\gamma\leq\kappa$ is inaccessible and $f"\gamma\subseteq\gamma$ then $\dot{\Q}_\gamma$ is a $\P_\gamma$-name for the forcing to add a single Cohen subset to $\gamma$, and otherwise $\dot{\Q}_\gamma$ is a $\P_\gamma$-name for trivial forcing. Let $G_{\kappa+1}\cong G_\kappa*H_\kappa\subseteq \P_\kappa*\dot{\Q}_\kappa$ be generic over $V$ and let $h_\kappa=\bigcup H_\kappa:\kappa\to 2$. 

We will identify each $z\in P_\kappa\lambda$ with a function $z:\ot(z)\to\lambda$ enumerating its elements in increasing order; in other words, $z(\alpha)$ denotes the $\alpha$-th element of $z$ where $\alpha<\ot(z)$. For each $z\in P_\kappa\lambda$ with $z\cap\kappa\in \kappa$ we define $a_z=\{z(\beta)\st\beta<\ot(z)\land h_\kappa(z\cap\kappa+\beta)=1\}$ and let $\vec{a}=\<a_z\st z\in P_\kappa\lambda\>$. Standard arguments show that $\P_{\kappa+1}$ preserves cofinalities under $\GCH$, so it remains to show that, in $V[G_{\kappa+1}]$, $W$ remains weakly compact and that $\<a_z\st z\in P_\kappa\lambda\>$ is a weakly compact diamond sequence on $W$.

Fix $X\in P(\lambda)^{V[G_{\kappa+1}]}$. It will suffice to show that $E_X(W)=\{z\in W\st X\cap z=a_z\}\in\NWC_{\kappa,\lambda}^{V[G_{\kappa+1}]}$. Fix $A\in P(\lambda)^{V[G_{\kappa+1}]}$ and let $\dot{A},\dot{X},\dot{E}_{X}(W),\dot{h}_\kappa,\dot{\vec{a}}\in H_{\lambda^+}^V$ be $\P_{\kappa+1}$-names for the appropriate sets. We assume that $\dot{A}$ and $\dot{X}$ are nice names for subsets of $\lambda$. Working in $V$, let $M\models\ZFC^{-}$ be transitive of size $\lambda$ with $\lambda,\dot{A},\dot{X},\dot{E}_X(W),\dot{h}_\kappa,\dot{\vec{a}},\P_{\kappa+1},\ldots\in M$. Since $W$ is weakly compact in $V$ there is a $j:M\to N$ such that $\crit(j)=\kappa$, $j(\kappa)>\lambda$, $j(f)(\kappa)>\lambda$ and $j"\lambda\in j(W)$, as in Lemma \ref{lemma_normal_embedding}.

We now show that standard arguments allow us to lift $j$ to have domain $M[G_\kappa]$. Since $N^{<\lambda}\cap V\subseteq N$ we have $j(\P_\kappa)\cong \P_\kappa*\dot{\Q}_\kappa*\dot{\P}'_{\kappa,j(\kappa)}$ where $\dot{\P}'_{\kappa,j(\kappa)}$ is a $\P_\kappa*\dot{\Q}_\kappa$-name for the tail of the iteration $j(\P_\kappa)$ as defined in $N$. Since $j(f)(\kappa)>\lambda$, it follows that the next stage of nontrivial forcing in $j(\P_\kappa)$ after $\kappa$ occurs beyond $\lambda$. Thus, it follows that in $N[G_\kappa*H_\kappa]$, the forcing $\P'_{\kappa,j(\kappa)}=_{\defn}(\dot{\P}'_{\kappa,j(\kappa)})_{G_\kappa*H_\kappa}$ is ${<}\lambda$-closed. Since $|N[G_\kappa*H_\kappa]|^{V[G_\kappa*H_\kappa]}\leq\lambda$, the poset $\P_{\kappa,j(\kappa)}$ has at most $\lambda$-dense sets in $N[G_\kappa*H_\kappa]$. The model $N[G_\kappa*H_\kappa]$ is closed under ${<}\lambda$-sequences in $V[G_\kappa*H_\kappa]$ and thus, working in $V[G_\kappa*H_\kappa]$ we may build a filter $G'_{\kappa,j(\kappa)}\in V[G_\kappa*H_\kappa]$ which is generic for $\P'_{\kappa,j(\kappa)}$ over $N[G_\kappa*H_\kappa]$. Since the critical point of $j$ is $\kappa$ it follows that $j[G_\kappa]\subseteq G_\kappa*H_\kappa*G'_{\kappa,j(\kappa)}$ and thus $j$ lifts to $j:M[G_\kappa]\to N[\hat{G}_{j(\kappa)}]$ where $\hat{G}_{j(\kappa)}=G_\kappa*H_\kappa*G'_{j(\kappa)}$.

We define $m:\lambda\to 2$ by letting $m\restrict\kappa=h_\kappa$ and $m(\kappa+\alpha)=X(\alpha)$ for $\kappa+\alpha<\lambda$, where we are identifying $X$ with it's characteristic function; that is, $X(\xi)=1$ if and only if $\xi\in X$. Since $h_\kappa$ is clearly in $N[\hat{G}_{j(\kappa)}]$, to check that $m\in j(\Q_\kappa)$ it will suffice to show that $X\in N[\hat{G}_{j(\kappa)}]$. Since $\dot{X}$ is a nice $\P_{\kappa+1}$-name for a subset of $\lambda$, it follows that $\dot{X}=\bigcup_{\alpha<\lambda}\{\alpha\}\times A_\alpha$ where $A_\alpha$ is an antichain of $P_{\kappa+1}$ for each $\alpha<\lambda$. Since $j"\lambda\in N$ we have $j"\dot{X}=\bigcup_{\alpha<\lambda}\{j(\alpha)\}\times j"A_\alpha\in$, and since $j\restrict \P_{\kappa+1}\in M$, it follows that $\dot{X}\in M$. Hence we have $X=\dot{X}_{G_{\kappa+1}}\in N[\hat{G}_{j(\kappa)}]$.

Since $m$ is a condition in $j(\Q_\kappa)$ we may build a filter $\hat{H}_{j(\kappa)}\subseteq j(\Q_\kappa)$ with $m\in\hat{H}_{j(\kappa)}$ which is generic over $N[\hat{G}_{j(\kappa)}]$. Since $j"H_\kappa\subseteq \hat{H}_{j(\kappa)}$, we may lift $j$ to $j:M[G_{\kappa}*H_\kappa]\to N[\hat{G}_{j(\kappa)}*\hat{H}_{j(\kappa)}]$. 

Since $\vec{a}\in M[G_\kappa*H_\kappa]$ we may let $j(\vec{a})=\<\bar{a}_z\st z\in j(P_\kappa\lambda)\>$. 
Since $m\in \hat{H}_{j(\kappa)}$, it follows that $j(h_\kappa)(\kappa+\beta)=X(\beta)$ for all $\beta<\lambda$. By definition $a_z=\{z(\beta)\st\beta<\ot(z)\land h_\kappa(z\cap\kappa+\beta)=1\}$, thus by elementarity 
\begin{align*}
j(\vec{a})(j"\lambda)=\bar{a}_{j"\lambda}&=\{(j"\lambda)(\beta)\st\beta<\ot(j"\lambda)\land j(h_\kappa)(j"\lambda\cap j(\kappa)+\beta)=1\}\\
	&=\{j(\beta)\st\beta<\lambda\land j(h_\kappa)(\kappa+\beta)=1\}\\
	&=\{j(\beta)\st\beta<\lambda\land X(\beta)=1\}\\
	&=j(X)\cap j"\lambda.
\end{align*}
Thus $j"\lambda\in j(E_X(W))$.
\end{proof}

Standard arguments can be used to prove the following.

\begin{proposition}
$\wcdiamond_{\kappa,\lambda}(W)$ implies that $\NWC_{\kappa,\lambda}\restrict W$ is not $\lambda$-saturated.
\end{proposition}

\subsection{Indescribable semimorasses}\label{section_semimorasses}

For cardinals $\kappa\leq\lambda$, a $(\kappa,\lambda)$-semimorass is a subset $\mu\subseteq P_\kappa\lambda$ which is well-founded with respect to $\subsetneq$ and satisfies certain homogeneity properties (see Definition \ref{definition_semimorass} below). The reader may consult \cite{MR1320107} or \cite{MR3696070} for some additional information and applications of semimorasses. If $\mu\subseteq P_\kappa\lambda$ and $x\in P_\kappa\lambda$, we define
\[\mu\restrict x=\{y\in \mu\st y\subsetneq x\}.\]

\begin{definition}\label{definition_semimorass}
Let $\kappa$ and $\lambda$ be cardinals with $\kappa\leq\lambda$. A \emph{$(\kappa,\lambda)$-semimorass} is a family $\mu\subseteq P_\kappa\lambda$ which satisfies the following properties.
\begin{enumerate}
\item $\mu$ is well-founded with respect to $\subsetneq$.
\item $\mu$ is \emph{locally small}, that is, for all $x\in\mu$, $|\mu\restrict x|<\kappa$.
\item $\mu$ is \emph{homogeneous}, that is, if $x,y\in\mu$ and $\rk(x)=\rk(y)$ then $x$ and $y$ have the same order type and $\mu\restrict y=\{f_{x,y}[z]\st z\in\mu\restrict x\}$, where $f_{x,y}:x\to y$ is the unique order-preserving isomorphism from $x$ to $y$.
\item $\mu$ is directed with respect to $\subseteq$, that is, for all $x,y\in \mu$ there is $z\in \mu$ such that $x,y\subseteq z$.
\item $\mu$ is \emph{locally semi-directed}, that is, for all $x\in\mu$ either
\begin{enumerate}
\item $\mu\restrict x$ is directed, or
\item there are $x_1,x_2\in\mu$ such that $\rk(x_1)=\rk(x_2)$ and $x=x_1*x_2$, that is, $x$ is the amalgamation of $x_1$ and $x_2$ with respect to $\mu$.
\end{enumerate}
\item $\mu$ \emph{covers} $\lambda$, that is, $\bigcup\mu=\lambda$.
\item $\mu$ has height $\kappa$.
\end{enumerate}
\end{definition}

Koszmider \cite{MR1320107} proved that if $\kappa^{<\kappa}=\kappa$ then there is a ${<}\kappa$-directed closed $\kappa^+$-c.c. poset $\K_{\kappa,\kappa^+}$ such that if $g\subseteq \K_{\kappa,\kappa^+}$ is a generic filter over $V$ then $\mu=\bigcup g$ is a stationary $(\kappa,\kappa^+)$-semimorass. Pereira \cite{MR3640048} used an Easton-support iteration of Koszmider forcings to prove that if $\kappa$ is $\kappa^+$-supercompact and $\GCH$ holds then there is a cofinality-preserving forcing extension $V[G]$ in which there is a normal fine $\kappa$-complete nonprincipal ultrafilter $U\in V[G]$ on $(P_\kappa\kappa^+)^{V[G]}$ which contains a $(\kappa,\kappa^+)$-semimorass $\mu\in U$. It is an easy consequence of Corollary \ref{corollary_normal_measure_one_sets_are_weakly_compact} and Pereira's result that if $\kappa$ is $\kappa^+$-supercompact and $\GCH$ holds then there is a cofinality-preserving forcing extension in which there is a $(\kappa,\kappa^+)$-semimorass $\mu\subseteq P_\kappa\kappa^+$ which is $\Pi^1_n$-indescribable for all $n<\omega$.

\begin{corollary}\label{corollary_weakly_compact_semimorass}
If $\kappa$ is $\kappa^+$-supercompact and $\GCH$ holds then there is a cofinality-preserving forcing extension $V[G]$ in which there is a $(\kappa,\kappa^+)$-semimorass $\mu\subseteq P_\kappa\lambda$ which is $\Pi^1_n$-indescribable for all $n<\omega$.
\end{corollary}

\section{Questions}\label{section_questions}

\subsection{Shooting $1$-clubs}\label{subsection_shooting_1_clubs}

In many ways $1$-club shooting forcings, and more generally $n$-club shooting forcings, are more well-behaved than club shooting. For example, Hellsten proved \cite{MR2026390} that if $W\subseteq\kappa$ is any weakly compact set and $\GCH$ holds then there is a cofinality-preserving forcing extension in which $W$ contains a $1$-club and all weakly compact subsets of $W$ are preserved, whereas the forcing to shoot a club through a given stationary set $S\subseteq\kappa$ may collapse cardinals unless $S$ contains the singular cardinals less than $\kappa$. Similarly, in \cite{CodyGitmanLambieHanson} it is shown that if $W\subseteq \kappa$ is any $\Pi^1_n$-indescribable set and $\GCH$ holds then there is a cofinality-preserving forcing extension in which $W$ contains an $n$-club and all $\Pi^1_n$-indescribable subsets of $W$ are preserved. Can these results be generalized to the two-cardinal context?
\begin{question}\label{question_1_club_shooting}
Suppose $W\subseteq P_\kappa\lambda$ is weakly compact and $\GCH$ holds. Is there a cofinality-preserving forcing extension in which $W$ contains a $1$-club and all weakly compact subsets of $W$ remains weakly compact?
\end{question}

The work of Gitik \cite{MR820120} seems to be relevant to answering Question \ref{question_1_club_shooting}, however this remains open. When attempting to answer Question \ref{question_1_club_shooting}, the author considered various Easton-support iterations, and in an attempt to build master conditions for such forcings, the author was led to the following related questions (Question \ref{question_weird_weakly_compact_set} and Question \ref{question_special_semimorass} below).

\begin{question}\label{question_weird_weakly_compact_set}
Is it consistent that there is a weakly compact set $W\subseteq P_\kappa\kappa^+$ which does not contain a $1$-club and 
\[\text{for all $x\in W$ if $y\subseteq x$ and $|y|\geq\kappa_x$ then $y\in W$?}\]
\end{question}

Koszmider (see \cite[Proposition 4.3]{MR3696070} or \cite[Proposition 10]{MR1320107}) has shown that if $\mu$ is a $(\kappa,\lambda)$-semimorass (see Definition \ref{definition_semimorass} above) then it satisfies the following non-reflection property:
\[\text{for every proper subset $X\subsetneq\lambda$ with $|X|\geq\kappa$ we have $\mu\cap P_\kappa X\in\NS_{\kappa,X}$.}\tag{K}\]
Combined with Corollary \ref{corollary_weakly_compact_semimorass}, this shows the consistency of the existence of a set $W\subseteq P_\kappa\kappa^+$ which is $\Pi^1_n$-indescribable for all $n<\omega$ and which satisfies Koszmider's non-reflection property (K). When attempting to build master conditions for various forcings, the author arrived at the following question concerning additional non-reflection properties of semimorasses.

\begin{question}\label{question_special_semimorass}
Is it consistent that there is a weakly compact set $W\subseteq P_\kappa\kappa^+$ such that $W$ is a \emph{special} $(\kappa,\kappa^+)$-semimorass, meaning that it is a $(\kappa,\kappa^+)$-semimorass and it satisfies the non-reflection property: for every proper subset $X\subsetneq\kappa^+$ with $|X|\geq\kappa$ we have $\mu\cap P_{\kappa_X}X\in\NS_{\kappa_X,X}$?
\end{question}

\subsection{The $P_\kappa\lambda$-weakly compact reflection principle}\label{subsection_reflection}

Schanker proved that if $\kappa$ is $\kappa^+$-supercompact cardinal and $\GCH$ holds then there is a forcing extension in which $\kappa$ remains nearly $\kappa^+$-supercompact and the $\GCH$ fails first at $\kappa$, hence $\kappa$ is not $\kappa^+$-supercompact or even measurable. Can we obtain a similar forcing result while preserving $\GCH$?

\begin{question}\label{question_GCH_and_near_supercompactness}
If $\kappa$ is $\kappa^+$-supercompact and $\GCH$ holds, is there a cofinality-preserving forcing extension in which $\kappa$ remains nearly $\kappa^+$-supercompact, $\GCH$ holds and $\kappa$ is not $\kappa^+$-supercompact? Phrased in our preferred terminology: if $P_\kappa\kappa^+$ is weakly compact and $\GCH$ holds, is there a cofinality-preserving forcing extension $V[G]$ in which $(P_\kappa\kappa^+)^{V[G]}$ is weakly compact, $\GCH$ holds and $\kappa$ is not $\kappa^+$-supercompact?
\end{question}

In order to address this question, let us use the following definition.

\begin{definition}
Suppose $\kappa$ is inaccessible and $X$ is a set of ordinals with $\kappa\leq|X|$ and $|X|^{<\kappa}=|X|$. We say that a set $W\subseteq P_\kappa X$ is \emph{weakly compact} if and only if it is $\Pi^1_1$-indescribable as a subset of $P_\kappa X$ via Definition \ref{definition_indescribable}. 
\end{definition}

Under the assumptions of the previous definition, the collection $\NWC_{\kappa,X}=\Pi^1_1(\kappa,X)$ is a strongly normal ideal on $P_\kappa X$ (see Section \ref{section_indescribability}). If $\kappa$ is $\kappa^+$-supercompact, then it follows that for every weakly compact $W\subseteq P_\kappa\kappa^+$ there is an $x\in P_\kappa\kappa^+$ such that $W\cap P_{\kappa_x}x\in\NWC_{\kappa_x,x}^+$.

\begin{definition}
For cardinals $\kappa\leq\lambda$, we say that $W\subseteq P_\kappa\lambda$ is a \emph{non-reflecting weakly compact set} if and only if $W$ is weakly compact and for all $x\in P_\kappa\lambda$ the set $W\cap P_{\kappa_x}x$ is not a weakly compact subset of $P_{\kappa_x}x$. We say that the  \emph{$P_\kappa\lambda$-weakly compact reflection principle} holds and write $\Refl_{\WC}(\kappa,\lambda)$ if and only if every weakly compact $W\subseteq P_\kappa\lambda$ reflects at some $x\in P_\kappa\lambda$.
\end{definition}

Hence, one could answer Question \ref{question_GCH_and_near_supercompactness} in the affirmative by showing that if $\kappa$ is $\kappa^+$-supercompact then there is a forcing extension $V[G]$ in which there is a non-reflecting weakly compact subset of $(P_\kappa\kappa^+)^{V[G]}$ and $(P_\kappa\kappa^+)^{V[G]}$ is weakly compact. However, it seems that subtle issues involved with building master conditions prevent one from using the usual forcing techniques.

\begin{question}
Suppose $P_\kappa\kappa^+$ is weakly compact and $\GCH$ holds. Is there a forcing extension $V[G]$ in which  $(P_\kappa\kappa^+)^{V[G]}$ is weakly compact, cofinalities are preserved and there is a weakly compact set $W\subseteq (P_\kappa\kappa^+)^{V[G]}$ such that for all $x\in (P_\kappa\kappa^+)^{V[G]}$ we have $(W\cap P_{\kappa_x}x\notin \NWC_{\kappa_x,x}^+)^{V[G]}$?
\end{question}

%Let me give a little more detail in this direction. As discussed above in Definition \ref{definition_wc_subset_of_P_kappa_X}, if $X$ is a set with $|X|=\lambda$ and $\kappa\leq\lambda$ is a cardinal then our definition of weakly compact subset of $P_\kappa\lambda$ leads naturally to a definition of weakly compact subset of $P_\kappa X$ using a bijection $b:X\to\lambda$. Thus we may consider the non--weakly compact ideal $\NWC_{\kappa,X}$ on $P_\kappa X$. Similarly, $C\subseteq P_\kappa X$ is $1$-club if and only if $\{b[x]\st x\in C\}$ is a $1$-club subset of $P_\kappa\lambda$. It easily follows that $W\subseteq P_\kappa X$ is weakly compact if and only for every $1$-club $C\subseteq P_\kappa X$ we have $W\cap C\neq\emptyset$, etc.

The proof of the following is a straightforward application of the fact that $\Pi^1_1$-indescribability can be expressed by a $\Pi^1_2$-sentence (see Lemma \ref{lemma_sentence_ind}).

\begin{proposition}\label{proposition_refl}
Suppose $\kappa\leq\lambda$ are cardinals with $\lambda^{<\kappa}=\lambda$ and $P_\kappa\lambda$ is $\Pi^1_2$-indescribable. Then $\Refl_{\WC}(\kappa,\lambda)$ holds.
\end{proposition}

\begin{question}
Is it consistent that $P_\kappa\lambda$ is weakly compact where $\kappa\leq\lambda$ and $\lambda^{<\kappa}=\lambda$, $\Refl_{\WC}(\kappa,\lambda)$ holds and $P_\kappa\lambda$ is not $\Pi^1_2$-indescribable?
\end{question}

\subsection{Alternative $1$-clubs}\label{subsection_1_clubs}

\begin{question}
One can formulate a notion of $1$-club subset of $P_\kappa\lambda$ using Jech's $\NS_{\kappa,\lambda}$ instead of $\NSS_{\kappa,\lambda}$. In other words, we define $C\subseteq P_\kappa\lambda$ to be Jech-$1$-club if and only if $C\in\NS_{\kappa,\lambda}^+$ and whenever $x\in P_\kappa\lambda$ is such that $\kappa_x$ is inaccessible and $C\cap P_{\kappa_x}x\in\NS_{\kappa_x,x}^+$ we have $x\in C$. What is the relationship between $1$-club and $1'$-club subsets of $P_\kappa\lambda$ when $\kappa$ is Mahlo?
\end{question}

%\bibliography{../../Bib/masterbib}
%\bibliographystyle{alpha}

\end{document}